\documentclass[11pt]{article}
\usepackage{amsmath,mathtools,amsfonts,amsthm,setspace,xcolor,mathrsfs,amssymb,a4wide,mathpazo,array,ulem}
\usepackage[textwidth=15cm]{geometry}
\usepackage[textwidth=15cm]{todonotes}
\usepackage{tikz}
\usetikzlibrary{arrows,shapes,decorations.markings,arrows.meta,decorations.pathmorphing,decorations.pathreplacing,quotes,positioning,tikzmark}
\tikzstyle{mybox} = [draw=red, fill=blue!20, very thick,
    rectangle, rounded corners, inner sep=10pt, inner ysep=20pt]

\newcommand{\N}{\mathbb{N}}
\newcommand{\R}{\mathbb{R}}

\newcommand{\ep}{\varepsilon}
\newcommand{\Eqref}[1]{Eq.\,(\ref{#1})}

\newcommand{\dN}{{\bf N}}

\newcommand{\dR}{{\bf R}}
\newcommand{\dZ}{{\bf Z}}

\makeatletter
\renewenvironment{proof}[1][\proofname] {\par\pushQED{\qed}\normalfont\topsep6\p@\@plus6\p@\relax\trivlist\item[\hskip\labelsep\bfseries#1\@addpunct{.}]\ignorespaces}{\popQED\endtrivlist\@endpefalse}
\makeatother

\newtheorem{theorem}{Theorem}
\newtheorem{corollary}[theorem]{Corollary}

\newtheorem{remark}[theorem]{Remark}

\newtheorem*{claim*}{Claim}

\newtheorem{lemma}[theorem]{Lemma}
\newtheorem{definition}[theorem]{Definition}
\newtheorem{example}[theorem]{Example}

\begin{document}

\title{Games with infinite past\footnote{We are grateful for useful discussions and suggestions to Dietmar Berwanger, Roberto Cominetti, Michael Greinecker, Christopher Kops, Abraham Neyman, Andr\'{e}s Perea, Arkadi Predtetchinski, J\'{e}r\^ome Renault, Marco Scarsini, Xavier Venel, Anna Zseleva, and the audiences of the MLSE seminar series and the conference 30 Years of Game Theory at Institut Henri Poincar\'{e}.
Solan acknowledges the support of the Israel Science Foundation, Grant \#211/22.}}
\author{Galit Ashkenazi-Golan\footnote{London School of Economics and Political Science, Houghton Street, London WC2A 2AE, UK, E-mail: galit.ashkenazi@gmail.com.}\and 
J\'{a}nos Flesch\footnote{Department of Quantitative Economics, 
Maastricht University, P.O.Box 616, 6200 MD, The Netherlands. E-mail: j.flesch@maastrichtuniversity.nl.}\and 
Eilon Solan\footnote{School of Mathematical Sciences, Tel-Aviv University, Tel-Aviv, Israel, 6997800, E-mail: eilons@tauex.tau.ac.il.}}

\maketitle

\begin{abstract}
We study multi-player games with perfect information and general payoff function, where the set of stages is the set of non-positive integers $\{\ldots,-2,-1,0\}$. We define two related equilibrium concepts: one  considering only deviations at finitely many stages and another considering all deviations. 
We show that (i) The sets of equilibrium plays coincide for the two equilibrium concepts, provided that at least two players are active along each infinite play. (ii) In win-lose games, the game has an equilibrium if the winning sets have Borel-rank at most 2, and we provide a counter-example showing that this is no longer true for Borel-rank 3. 
(iii) In general non-zero-sum games, the game has an equilibrium if 
the payoff functions are continuous, for example, with reversed-time discounted payoffs. The challenge for all these results is that not all strategy profiles admit a consistent infinite play, hampering the use of backward induction arguments.\medskip
\end{abstract}

\section{Introduction}

Dynamic games form a standard model to describe and analyze repeated interaction between multiple decision makers, and they are used regularly in various disciplines, among others in economics, computer science, biology, psychology, and descriptive set theory. Oftentimes, the interaction takes place on a very large, possibly unknown or clouded, number of occasions, making it virtually impossible or even undesirable to set a specific future time horizon in the model description. To capture such situations, it is common to employ mathematical abstraction and assume that the number of stages of the dynamic game is actually infinite, for instance, the set of natural numbers.  
In the past 80 years,
infinitely repeated games have been extensively studied;
they provide a tractable framework to analyzing various aspects of dynamic interactions and offer elegant solutions to a diverse set of applications.

In the very same vein, interactions regularly have a long past, or a clouded starting point. When did the human race begin its interaction with its environment? Was it perhaps $300,000$ years ago, when the first fossils of homo sapiens were found,
or was it when homo sapiens sufficiently differed anatomically from its predecessor species? 
And what does ``sufficiently differed anatomically'' mean anyway? And was there an exact time when the competition between Coca-Cola and Pepsi started? Was it in $1893$, when Pepsi was invented, or in $1898$, when this beverage was branded Pepsi-Cola, or in the early $1900$s, when both firms started national distribution? 

To describe and analyze such interactions, in this paper we propose a model in which the set of stages has no initial element, or putting it differently, the game has an infinite past. We view this paper as a first step in this direction and hope that the many interesting questions it raises will inspire follow-up work.\medskip

\noindent\textbf{Our model.} 
We study games with perfect information that have no starting point, yet they do have a definite end point. 
Formally, these games are played on the set of stages $\dZ_0 := \{\dots, -2,-1,0\}$.
Thus, similarly to standard repeated games,
these games last infinitely many stages,
but unlike repeated games, these infinitely many stages lie in the past and not in the future. That is, any stage of the game is preceded by an infinite number of stages, implying that the action choice of a player in any stage may depend on the infinitely many actions selected by the players in the earlier stages.

In our games, strategies can be defined in the usual way. Yet, as the game has no initial stage and no initial position, a strategy profile does not always determine a unique consistent run, i.e., an infinite sequence of actions that follows the actions recommended by the strategies. 
In fact, an interesting feature of games with infinite past
is that
there are strategy profiles that do not admit a consistent run at all. Indeed, suppose that there is only one player, whose strategy recommends action 0 if action 1 was selected in all previous stages, and recommends action 1 otherwise (cf.~Example \ref{ex-norun}). 
It can be seen 
that no run is consistent with this strategy. An important implication of this phenomenon is that the use of backward induction is largely hampered in our games: a strategy profile determined by backward induction may not be a solution of the game because it 
might admit no consistent run (cf.~Section \ref{sec-backind} for examples, and for a more detailed discussion). 

Therefore, in most of our existence proofs, we replace backward induction by Cantor's intersection theorem (cf.~Section \ref{sec-backind} for more details), which states that in a Hausdorff topological space, every nested sequence of nonempty and compact sets has a nonempty intersection. More concretely, in the set of all runs, we construct a nested sequence of nonempty and compact sets, such that each set in this sequence only contains runs that behave nicely from a given stage. 
Cantor's intersection theorem will
imply that the intersection of these sets is nonempty, and 
then 
every run in the intersection can be supplemented with a strategy profile so that we obtain an equilibrium. 

It is important to note that there are also strategy profiles with multiple consistent runs (cf.~Example \ref{ex-tworuns} and Remark \ref{rem-contcons}), so any solution concept must also be careful in how it  compares strategy profiles from a player's point of view.\medskip

\noindent\textbf{Main results.}\\
\textbf{[1] Equilibrium concepts:} Motivated by the considerations above, we define an equilibrium concept based not only on a strategy profile, but also on a consistent run. We call a pair $(s,r)$, consisting of a strategy profile $s$ and a run $r$ that is consistent with $s$, an \textit{equilibrium} (cf.~Definition \ref{def-weakeq}) if in each stage, in the corresponding position along $r$, the strategy profile $s$ induces a Nash equilibrium in the continuation subgame. Thus, the concept of an equilibrium only takes into account deviations from the run $r$ in a finite number of stages. 

We also define a closely related but stronger concept, which is solely based on a strategy profile. We call a strategy profile $s$ a \textit{strong equilibrium} (cf.~Definition \ref{def-strongeq}) if $s$ admits a unique consistent run $r$, and for any player, her payoff in $r$ is at least as much as her payoff in any run
that is consistent if she deviates from her strategy. Thus, the concept of strong equilibrium also considers deviations in infinitely many stages.

We explore the properties of these two equilibrium concepts, and their  relation. The definitions directly imply that if $s$ is a strong equilibrium and $r$ is the (unique) run consistent with $s$, then $(s,r)$ is an equilibrium. We prove a converse under a mild condition: if along each run at least two players are active infinitely often, then for every equilibrium $(s,r)$, there is a strong equilibrium $s'$ such that $r$ is the unique run that is consistent with $s'$ (cf.~Theorem \ref{theorem-weak=strong}).
This shows that the two notions of equilibrium are highly interconnected.

The set of runs in the game can be partitioned in a natural way: two runs belong to the same element of the partition if they coincide up to some stage, i.e., 
they have a common 
(infinite) prefix.
We call each element of this partition a \textit{segment} of the set of runs (cf.~Definition \ref{def-seg}). 
For example, if $a$ and $b$ are two actions, then the run consisting only of action $a$ and the run consisting only of action $b$ have no common 
prefix,
and therefore belong to different segments. 
We discuss the equilibrium concepts when the game is restricted to a segment (cf.~Section \ref{sec-seg}). As it turns out, within a segment, the game behaves much closer to ``standard'' games.\smallskip

\noindent\textbf{[2] Two-player win-lose games.} Zero-sum games is one of the most important classes of games, and win-lose games are arguably the most basic versions of them. 
A two-player game is called a win-lose game if there are only two outcomes: 
a win for player~1 (which is a loss for player~2)
and a win for player~2 (which is a loss for player~1).
Two-player win-lose games are frequently studied in the game-theoretic literature, and they play a major role in the field of descriptive set theory (cf.~Kechris [2012]). 

In ``standard'' two-player win-lose games where the set of stages is the set of natural numbers, a landmark result of Martin [1975] shows that an equilibrium exists as soon as 
player~1's winning set
is Borel-measurable (with respect to the product topology). This result is a building stone of many existence results, cf.~for instance, the proof of Mertens and Neyman in Mertens [1987], 
Chatterjee and Henzinger [2012], 
Martin [1998], Maitra and Sudderth [1998], Ashkenazi-Golan et al.~[2022], Flesch and Solan [2023], and the references therein. 

In this paper we explore the existence of equilibrium in two-player win-lose games in the presence of infinite past. 
We show 
that in such games,
a (strong) equilibrium exists as long as player 1's winning set has Borel-rank at most 2 (cf.~Theorem \ref{theorem-rank2}), but it 
need not exist
for winning sets of Borel-rank 3 or higher (cf.~Example \ref{ex-counter}). We also examine the consequences of these results when the game is restricted to a segment (cf.~Section \ref{sec-segwinlose}), and we find a much closer parallel to ``standard'' games, e.g., in a segment one can define the value in the usual way.\smallskip

\noindent\textbf{[3] Multi-player (non-zero-sum) games.} We examine multi-player games with infinite past. In ``standard'' games where the set of stages is the set of natural numbers, some of the most common existence results are based on continuity of the payoff functions (e.g., Fudenberg and Levine [1983]). We examine games with continuous payoffs in the presence of infinite past, and establish the existence of (strong) equilibrium, provided that the identity of the active player is also assigned in a continuous way to the set of positions (cf.~Theorem \ref{theorem-cont}).\medskip

\noindent\textbf{Related literature.} 
As mentioned before, our paper is related to the literature on infinite horizon games, where the set of stages is the set of natural numbers, see for instance Fudenberg and Tirole [1991], Mertens et al.~[2015], Maschler et al. [2020], and to the literature on the determinacy of two-player win-lose games, which started with the seminal paper of Gale and Stewart [1953], see for instance Martin [1975] and Kechris [2012]. 

To the best of our knowledge, there is only a limited amount of literature on games where the past is infinite. 
There is a closely related discussion on the forum MathOverflow,\footnote{https://mathoverflow.net/questions/112150/games-that-never-begin.} which focuses on two-player win-lose games with infinite past, and includes various ideas and questions: (1) The question is raised on how to define optimal strategies in such a game, and some ideas are proposed. (2) Consistency of runs with strategies is discussed, including the strategy in Example \ref{ex-norun}, which admits no consistent run. (3) It is mentioned that it would be natural to define winning strategies in such a way that they are winning in each subgame where this player has a winning strategy. This idea is also reflected in our equilibrium notion. (4) Relations to infinite time computation and infinitary logic is discussed. (5) Using the axiom of choice, a game is described that admits no value. As far as we can see, this idea does not guarantee that the winning set is Borel. 

We now mention a few other papers where some kind of infinite past appears. From a philosophical perspective, Morriston [1999] and Sorensen [1999]  raised the question whether the past must have a beginning, and Sorensen [1999] also discussed to which extent infinite backward induction arguments make sense. Algoet and Cover [1992] described an investment model, where infinite past can be interpreted as an idealized theoretical situation. Gorokhovsky and Rubinchik [2018] considered general equilibrium in a model with infinite past and future, in which the inclusion of translation-invariant equilibrium (fixed-point) mappings imply that the model has to include the infinite past. Puente [2006] considered discounted repeated games where the set of stages is the set of integers. His equilibrium notion requires that the strategy profile be an equilibrium from any stage on, in the remaining infinitely many stages, and among others, he presented a folk theorem for these games. 

We would also like to mention that various papers consider dynamic models where, in particular cases, an underlying Markov process starts in its invariant distribution, see, e.g., 
Kandori et al. [1993],
Iyer et al. [2012],
Adlakha et al. [2015],
or
Balseiro et al. [2015]. 
This may have the interpretation that the game starts after an infinite past, in which the Markov process reached its invariant distribution. 

Finally, for a general treatment of extensive form games, we refer to the book by Al\'{o}s-Ferrer and Ritzberger [2016].\medskip

\noindent\textbf{Structure of the paper.} In Section \ref{sec-prel}, we discuss some mathematical preliminaries. In Section \ref{sec-model}, we define our model, and discuss strategies and consistent runs, as well as the partitioning of the set of runs into segments. In Section \ref{sec-eq}, we define and compare our equilibrium concepts. Sections \ref{sec-winlose} and \ref{sec-multiplayer} are devoted to two-player win-lose games and to games with continuous payoffs, respectively. We close the paper with some concluding remarks in Section \ref{sec-discussion}.  

\section{Preliminaries - the Borel Hierarchy}\label{sec-prel}

Let $X$ be a metrizable topological space. The Borel hierarchy on $X$ consists of classes of Borel sets, describing how Borel sets can be constructed from open sets using complementation and countable unions. We will have specific results for the first three ranks of the Borel hierarchy. For a more general and detailed discussion we refer to Kechris [2012, p68].

The class $\Sigma_1$ consists of all open sets, and the class $\Pi_1$ consists of all closed sets. The sets belonging to $\Sigma_1$ or to $\Pi_1$ are said to have \emph{Borel-rank 1}. 

The class $\Sigma_2$ consists of all countable unions of sets in $\Pi_1$ (these are 
called the $F_\sigma$-sets), and the class $\Pi_2$ consists of the complements of the sets in $\Sigma_2$ (these are 
called the $G_\delta$-sets, and they can be written as countable intersections of open sets). The sets belonging to $\Sigma_2$ or to $\Pi_2$ are said to have \emph{Borel-rank 2}, and because $X$ is metrizable, $\Delta_2=\Sigma_2\cap \Pi_2$ contains all sets of rank 1.

The class $\Sigma_3$ consists of all countable unions of sets in $\Pi_2$, and the class $\Pi_3$ consists of the complements of the sets in $\Sigma_3$. The sets belonging to $\Sigma_3$ or to $\Pi_3$ are said to have 
\emph{Borel-rank 3}, and $\Delta_3=\Sigma_3\cap \Pi_3$ contains all sets of rank 2 (and rank 1).

\section{The model}\label{sec-model}

In this section we provide the formal definition of our model. 
In Section \ref{sec-game} we define games with infinite past. In Section \ref{sec-str} we define strategies and runs that are consistent with strategies. Finally, in Section \ref{sec-seg} we discuss a natural partition of the set of positions of the game into so-called segments, which gives a deeper insight into the strategic possibilities of the players and consistency of runs.

\subsection{The game}\label{sec-game}

Let $\dZ_0=\{\ldots,-2,-1,0\}$ be the set of non-positive integers. For each $n\in\dZ_0$, let $\dZ_n=\{\ldots,n-2,n-1,n\}$. We emphasize that the symbol $n$ will always represent an element of $\dZ_0$, and hence a non-positive integer.\medskip

We study dynamic games where before each stage of the game, already infinitely many earlier stages took place. In particular, in such a game there is no initial position. 
We formalize this property by setting the set of stages to be the set of non-positive integers $\dZ_0$. Thus, stage 0 is the terminal stage.%
\footnote{In Section~\ref{sec-discussion} we discuss games that are played over the set of stages $\dZ$ -- the set of integers.}

\begin{definition}[\textbf{Game with infinite past}]\label{def-game}\rm 
A \emph{game with infinite past}\ is a tuple $G=(I,A,\iota,(u_i)_{i\in I})$ where 
\begin{itemize}
\setlength\itemsep{0cm}
	\item $I$ is a nonempty finite set of \emph{players}.
	 \item $A$ is a nonempty finite set of \emph{actions} 
    that contains at least two elements.
\\[0.05cm]
	A \emph{position} in stage $n\in\dZ_0$ is a sequence
    $p_n \in P_n:=A^{\dZ_{n-1}}$,
    which is denoted\footnote{In our notation we will always include specifically the stage $n$, either as a subscript in $p_n$, or through the subscript in the final action $a_{n-1}$. Note that two positions in two different stages, while both being an infinite sequence of actions, are never equal, because the domains (the index sets) of the relevant stages are different.} 
    by $p_n=(\ldots,a_{n-2},a_{n-1})$. The set of positions is $P=\bigcup_{n\in\dZ_0}P_n$. 
    A \emph{run} is a sequence $r \in R :=A^{\dZ_{0}}$, which is denoted by $r=(\ldots,a_{-2},a_{-1},a_{0})$. 
    \color{black}
	\item $\iota:P\to I$ is a function that, to each position $p\in P$, assigns a player $\iota(p)$ who controls this position $p$.
	\item $u_i:R\to\dR$ is a bounded payoff function, for each player $i\in I$.
\end{itemize}
\end{definition}

\begin{remark}[\textbf{On the set of actions}]\rm 
The requirement that the same set of actions is available for all players at all positions is made for convenience only. $\blacklozenge$
\end{remark}

The interpretation of the game is as follows. If the current stage is $n\in\dZ_0$, then all stages $k<n$ are past stages and all stages $k>n$ are future stages.
If the current position is $p_n=(\ldots,a_{n-2},a_{n-1})$, then player $\iota(p_n)$ chooses an action $a_n\in A$, which is observed by all players. If $n<0$, then the game continues in stage $n+1$ in position $p_{n+1}=(\ldots,a_{n-2},a_{n-1},a_n)$. 
If $n=0$, then the game ends and the outcome of the game is the run $r=(\ldots,a_{-2},a_{-1},a_0)$ with payoff $u_i(r)$ to each player $i\in I$.
Thus, the position in a stage is the infinite sequence of actions selected by the players in all earlier stages. 
 
\begin{remark}[\textbf{Players' information in games with infinite past}]\rm 
Games with infinite past are perfect information games: in every stage, only one player chooses an action, and each player knows the actions that were selected in all previous stages of the game. $\blacklozenge$
\end{remark}

For every position $p_n=(\ldots,a_{n-2},a_{n-1})$ and
every
action $a_n$,
the position obtained by selecting $a_n$ at $p_n$ is denoted by $(p_n,a_n)= (\ldots,a_{n-2},a_{n-1},a_n)$.
For every run $r=(\ldots,a_{-2},a_{-1},a_0)$
and every $n \in \dZ_0$,  
the position in stage $n$ along $r$ is
denoted by $r_{<n} = (\ldots,a_{n-2},a_{n-1})$.\medskip

\noindent\textbf{Subgames.} Each position $p_n=(\ldots,a_{n-2},a_{n-1})\in P_n$ in stage $n\in\dZ_0$ naturally induces a \emph{subgame} $G[p_n]$: this is the continuation game played in the stages $n,n+1,\ldots,0$ that arises when up to stage $n-1$ the actions are given by $p_n$. We emphasize that a subgame is a finite game: it lasts only finitely many stages and the players have only finitely many actions.\medskip

\noindent\textbf{The product topology on the set $R$ of runs.} Following the literature on dynamic games and descriptive set theory, 
we endow the set of runs $R=A^{\dZ_0}$ with the product topology on $R$ when $A$ is endowed with its discrete topology. Note that $R$ with this topology is metrizable, so the Borel hierarchy of $R$ has the properties as discussed in Section~\ref{sec-prel}. Moreover, $R$ is compact because $A$ is finite.

\subsection{Strategies and consistent runs}\label{sec-str}

In games with infinite past, strategies are defined in the usual way, 
as a map from the set of positions where a player is active to the set of actions.
As is customary in perfect information games, we restrict attention to pure strategies.

\begin{definition}[\textbf{Strategy}]\label{def-str}\rm A \emph{strategy} for player $i\in I$ is a map $s_i\colon \{p\in P\colon \iota(p)=i\}\to A$ that assigns an action $s_i(p)$ to each position $p$ that player $i$ controls. A \emph{strategy profile} is a vector $s=(s_i)_{i\in I}$ consisting of one strategy for each player.
\end{definition}

In the standard setting where the set of stages is the set $\N$ of natural numbers, the game starts in a given initial position in stage 1, and therefore every strategy profile induces a unique run. In contrast, in games with infinite past there is no initial stage and no initial position, and therefore, under a strategy profile, it is a priori not always clear which run will be realized. 

In the definition below, we define consistency of a run with a strategy of a given player, which intuitively requires that this run not be excluded under this strategy. More precisely, a run is consistent with a strategy of a player if, along the run,
in each position that is controlled by this player, the strategy chooses the action corresponding to the run. We define similarly consistency of a run with a strategy profile.  

\begin{definition}[\textbf{Consistent run}]\rm A run $r=(\ldots,a_{-1},a_{0})\in R$ is \emph{consistent} with a strategy~$s_i$ of player $i\in I$ if for each stage $n\in\dZ_0$ for which $\iota(r_{<n})=i$ we have $a_{n} = s_{i}(r_{<n})$.
A run $r\in R$ is consistent with a strategy profile $s=(s_i)_{i\in I}$ if $r$ is consistent with the strategy $s_i$ for each player $i\in I$. We denote by $R(s)$ the set of runs consistent with a strategy profile~$s$.
\end{definition}

In games with infinite past,
sometimes a consistent run does not exist,
and sometimes 
there are several runs that are consistent with the same strategy.

\begin{example}[\textbf{A strategy with no consistent run}]\label{ex-norun}\rm 
Suppose that there is only one player and there are two actions: $I = \{1\}$ and $A = \{0,1\}$.
Consider the following strategy $s_1$: play action 0 if action 1 was selected in \emph{all} previous stages, and play action 1 otherwise.
That is,
for every $n\in\dZ_0$ and every position $p_n\in P_n$,
\[ s_1(p_n) = \left\{
\begin{array}{ll}
0, & \hbox{if } p_n = (\dots,1,1,1),\\
1, & \hbox{otherwise}.
\end{array}
\right.
\]
We briefly argue that $s_1$ admits no consistent run: $R(s_1)=\emptyset$. Indeed, suppose the opposite: the run $r\in R$ is consistent with $s_1$. If $r$ only contains action 1, then $s_1$ recommends action 0 in any stage along $r$, which is a contradiction. And if $r$ contains action 0 at least once, say in stage $n\in\dZ_0$, then due to consistency with $s_1$, the run $r$ only contains action~1 until stage $n-1$, which is again a contradiction, just as in the previous case. $\blacklozenge$
\end{example}

In the following example, the strategy admits exactly two consistent runs. As we discuss later in Remark \ref{rem-contcons}, a strategy can also admit a continuum of consistent runs. 

\begin{example}[\textbf{A strategy with two consistent runs}]\label{ex-tworuns}\rm 
Suppose that there is only one player and there are two actions: $I = \{1\}$ and $A = \{0,1\}$. Consider the strategy $s_1$ that always repeats the previous action: in any stage $n$, in any position $p_n=(\dots,a_{n-2},a_{n-1}) \in P_n$, we have $s_1(p_n) = a_{n-1}$. 

Then, exactly two runs are consistent with $s_1$: the run in which the player always chooses action $0$, and the run in which the player always chooses action $1$. That is, $R(s)=\{(\ldots,0,0),(\ldots,1,1)\}$. $\blacklozenge$
\end{example}

\begin{remark}[\textbf{On the set $R(s)$ of all consistent runs}]\rm
Given a strategy profile $s$, the set $R(s)$ of consistent runs can be determined through an infinite elimination procedure: Let $R^n(s)$ denote the set of runs $r$ such that $r$ is consistent with $s$ in stages $n,\ldots,-1,0$. Then, $R^n(s)$ is a non-increasing sequence, and $R(s)=\bigcap_{n=0}^{-\infty} R^n(s)$.\\
Note that every finite-memory strategy profile $s$, where the action choice only depends on the last $k$ stages, for some $k\in\N$, always admits a consistent run: $R(s)\neq\emptyset$. $\blacklozenge$
\end{remark}

\subsection{Segments}\label{sec-seg}

In this section we discuss a natural partition of the set of positions of the game into so-called segments. The main merit of 
this partition 
is that it provides a deeper insight into the structure of the game, the strategic possibilities of the players, and the consistency properties of runs with strategies and strategy profiles. Moreover, we will be able to define a game restricted to a segment, which has properties much closer to ``standard'' games. And, we will use segments later to explain our solution concepts and main results in more detail. 

We define a relation, denoted by $\sim$, on the set $R$ of runs: for two runs $r,r'\in R$, we write $p\sim p'$ if and only if they coincide until a certain stage: there is a stage $n\in\dZ_0$ such that $r_{<n}=r'_{<n}$. Equivalently, $r\sim r'$ if and only if $r$ and $r'$ differ in at most finitely many actions. The relation $\sim$ on $R$ is an equivalence relation (reflexive, symmetric, and transitive), and therefore it induces equivalence classes. 

\begin{definition}[\textbf{Segment and corresponding positions}]\label{def-seg}\rm
A set $\Omega\subseteq R$ of runs is called a \emph{segment} if $\Omega$ is an equivalence class with respect to the relation $\sim$ on $R$. The set of positions corresponding to a segment $\Omega$, denoted by $P_\Omega$, consists of those positions that are prefixes of a run in $\Omega$:
\[P_\Omega\,=\,\big\{p\in P\colon p=r_{<n}\text{ for some run }r\in\Omega\text{ and for some stage }n\in\dZ_0\big\}.\]
\end{definition}

Segments, and the set of positions corresponding to them, have the following properties:
\begin{itemize}
    \item Let $\Omega$ be a segment. Then, in any position $p\in P_\Omega$ it is already fixed that the run will belong to $\Omega$: it does not matter which actions are chosen in the subgame $G[p]$, the induced run will belong to $\Omega$. So, intuitively, the segment is already determined by the pre-tail of the run, that is, by the actions played in stages $\ldots,n-2,n-1,n$ where $n$ is arbitrarily small.
    \item The collection $\big\{P_\Omega\colon \Omega\text{ is a segment}\big\}$ is a partition of the set $P$ of positions.
    \item The cardinality of each segment $\Omega$ is countably infinite, and so is the cardinality of $P_\Omega$. Moreover, each segment $\Omega$ is a countable union of closed sets, so $\Omega$ has Borel-rank 2. 
    These two observations
    follow from the fact that for every segment $\Omega$ and every run $r\in\Omega$, we have $\Omega\,=\,\bigcup_{n\in\dZ_0} R(n,r)$, where $R(n,r)=\{r'\in R\colon r_{<n}=r'_{<n}\}$ is the finite set of runs that coincide with $r$ until stage $n$. 
    \item The cardinality of the collection of all segments is equal to the continuum, and so is the cardinality of the collection $\big\{P_\Omega\colon \Omega\text{ is a segment}\big\}$. 
    These observations hold because 
    the cardinality of the set $R$ of runs is equal to the continuum, and each segment is countably infinite.
\end{itemize}

Figure \ref{fig-seg} provides a visual illustration of the segments of a game.

\begin{figure}
    \centering
\vspace{0.5cm}
\begin{tikzpicture}[scale=0.45]

\node[right] at (-3.5,0) {};
\node[right] at (-3.5,1) {\tiny stage $0$};
\node[right] at (-3.5,2) {\tiny stage $-1$};

\node[circle, fill=black, inner sep = 2pt] (n11) at (0,0) {};
\node[circle, fill=black, inner sep = 2pt] (n12) at (1,0) {};
\node[circle, fill=black, inner sep = 2pt] (n13) at (2.4,0) {};
\node[circle, fill=black, inner sep = 2pt] (n14) at (3.4,0) {};
\node[circle, fill=black, inner sep = 2pt] (n21) at (0.5,1) {};
\node[circle, fill=black, inner sep = 2pt] (n22) at (2.9,1) {};
\node[circle, fill=black, inner sep = 2pt] (n31) at (1.7,2) {};
\node[circle, fill=black, inner sep = 1pt] at (4.5,0.5) {};
\node[circle, fill=black, inner sep = 1pt] at (4.8,0.5) {};
\node[circle, fill=black, inner sep = 1pt] at (5.1,0.5) {};
\node[circle, fill=black, inner sep = 1pt] at (2.75,2.9) {};
\node[circle, fill=black, inner sep = 1pt] at (3,3.1) {};
\node[circle, fill=black, inner sep = 1pt] at (3.25,3.3) {};
\node[circle, fill=black, inner sep = 1pt] at (4.5,1.5) {};
\node[circle, fill=black, inner sep = 1pt] at (4.8,1.5) {};
\node[circle, fill=black, inner sep = 1pt] at (5.1,1.5) {};
\draw[->] (n21) -- (n11);
\draw[->] (n21) -- (n12);
\draw[->] (n22) -- (n13);
\draw[->] (n22) -- (n14);
\draw[->] (n31) -- (n21);
\draw[->] (n31) -- (n22);

\draw[dashed, dash pattern=on 1.5pt off 1.5pt, line width = 1pt] (6.1,-0.3) -- (6.1,6);

\node[circle, fill=black, inner sep = 1pt] at (7,1.5) {};
\node[circle, fill=black, inner sep = 1pt] at (7.3,1.5) {};
\node[circle, fill=black, inner sep = 1pt] at (7.6,1.5) {};
\node[circle, fill=black, inner sep = 1pt] at (7.9,1.5) {};
\node[circle, fill=black, inner sep = 1pt] at (8.2,1.5) {};
\node[circle, fill=black, inner sep = 1pt] (ndts) at (8.5,1.5) {};
\node[circle, fill=black, inner sep = 1pt] at (8.8,1.5) {};
\node[circle, fill=black, inner sep = 1pt] at (9.1,1.5) {};
\node[circle, fill=black, inner sep = 1pt] at (9.4,1.5) {};
\node[circle, fill=black, inner sep = 1pt] at (9.7,1.5) {};
\node[circle, fill=black, inner sep = 1pt] at (10,1.5) {};
\node[circle, fill=black, inner sep = 1pt] at (10.3,1.5) {};
\node[circle, fill=black, inner sep = 1pt] at (10.6,1.5) {};
\node[circle, fill=black, inner sep = 1pt] at (10.9,1.5) {};
\node[circle, fill=black, inner sep = 1pt] at (11.2,1.5) {};
\node[circle, fill=black, inner sep = 1pt] at (11.5,1.5) {};
\node[right] at (6.5,0.9) {\tiny continuum of segments}; 

\draw[dashed, dash pattern=on 1.5pt off 1.5pt, line width = 1pt] (12.4,-0.3) -- (12.4,6);

\node[circle, fill=black, inner sep = 2pt] (k11) at (15.1,0) {};
\node[circle, fill=black, inner sep = 2pt] (k12) at (16.1,0) {};
\node[circle, fill=black, inner sep = 2pt] (k13) at (17.5,0) {};
\node[circle, fill=black, inner sep = 2pt] (k14) at (18.5,0) {};
\node[circle, fill=black, inner sep = 2pt] (k21) at (15.6,1) {};
\node[circle, fill=black, inner sep = 2pt] (k22) at (18,1) {};
\node[circle, fill=black, inner sep = 2pt] (k31) at (16.8,2) {};
\node[circle, fill=black, inner sep = 1pt] at (13.4,0.5) {};
\node[circle, fill=black, inner sep = 1pt] at (13.7,0.5) {};
\node[circle, fill=black, inner sep = 1pt] at (14,0.5) {};
\node[circle, fill=black, inner sep = 1pt] at (15.5,3.3) {};
\node[circle, fill=black, inner sep = 1pt] at (15.75,3.1) {};
\node[circle, fill=black, inner sep = 1pt] at (16,2.9) {};
\node[circle, fill=black, inner sep = 1pt] at (13.4,1.5) {};
\node[circle, fill=black, inner sep = 1pt] at (13.7,1.5) {};
\node[circle, fill=black, inner sep = 1pt] at (14,1.5) {};
\draw[->] (k21) -- (k11);
\draw[->] (k21) -- (k12);
\draw[->] (k22) -- (k13);
\draw[->] (k22) -- (k14);
\draw[->] (k31) -- (k21);
\draw[->] (k31) -- (k22);
\end{tikzpicture}
\caption{\small Segments of a game with two actions, $\ell$ (left) and $r$ (right). 
The left (resp., right) third of the figure corresponds to the segment that contains the run $(\ldots,\ell,\ell)$
(resp., $(\ldots,r,r)$).
The bottom-left node corresponds to the run $(\ldots,\ell,\ell)$, the node above it to the position $(\ldots,\ell,\ell)$ in stage $0$, the bottom-right node to the run $(\ldots,r,r)$,
etc. 
The dotted vertical lines indicate that segments have no common position, and between the two dotted lines there are a continuum of segments.}
\label{fig-seg}
\end{figure}
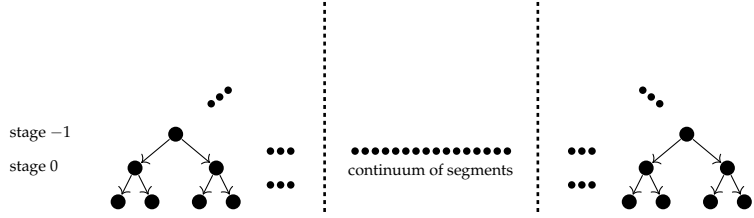

\begin{definition}[\textbf{Strategy permitting a segment}]\rm 
A strategy $s_i$ of player $i\in I$ is said to \emph{permit a segment} $\Omega$ if there is a run $r\in\Omega$ that is consistent with $s_i$.
\end{definition}

The next lemma specifies the number of consistent runs under a strategy profile depending on which segments are permitted by the players' strategies.

\begin{lemma}[\textbf{Number of consistent runs}]\label{lem-strpercons}
Consider a strategy profile $s=(s_i)_{i\in I}$ and a segment $\Omega$. 
\begin{enumerate}
    \item[{[1]}] If $s_i$ does not permit $\Omega$ for some player $i\in I$, then $s$ admits \emph{no} consistent run in $\Omega$.
    \item[{[2]}] If $s_i$ permits $\Omega$ for each player $i\in I$, then $s$ admits a \emph{unique} consistent run in $\Omega$, i.e., $R(s)\cap\Omega$ is a singleton. 
\end{enumerate}
\end{lemma}

\begin{proof} Part 1 follows from the definitions. Thus, it remains to prove Part~2. So, assume that $s_i$ permits $\Omega$ for each player $i\in I$. 

First we show that $R(s)\cap\Omega$ contains at most one run. Indeed, let $r,r'\in R(s)\cap\Omega$. As $r,r'\in \Omega$, there is a stage $n$ such that $r_{<n}=r'_{<n}$. 
Since
$r,r'\in R(s)$, 
it follows 
that $r=r'$. 

Now we show
that
$R(s)\cap\Omega$ contains at least one run. Since $s_i$ permits $\Omega$ for each player $i\in I$, each $s_i$ admits a consistent run $r_i\in\Omega$. For all $i,j\in I$ with $i\neq j$, as $r_i$ and $r_j$ belong to the same segment $\Omega$, there is a stage $n_{i,j}$ such that $r_i$ and $r_j$ coincide up to stage $n_{i,j}$. Let $n$ be the minimum of all such $n_{i,j}$. Then, all runs $r_i$, where $i\in I$, coincide with each other up to stage $n$. 
Let this common position in stage $n$ be denoted by $p_n=(\ldots,a_{n-2},a_{n-1})$. 
Consider the run $r=(p_n,a_n,\ldots,a_{-1},a_0)$ where the actions $a_n,\ldots,a_{-1},a_0$ are chosen according to the strategy profile $s$. By construction, $r$ is consistent with $s$ and $r\in \Omega$.
\end{proof}

In view of Example \ref{ex-norun}, a strategy of a given player may not permit any segment, and in view of Example \ref{ex-tworuns}, a strategy may permit some segments but not others. We will present a generalization of these statements below in Lemma \ref{lem-whichseg}.

We say that along a run $r=(\ldots,a_{-1},a_0)$ player $i\in I$ is \emph{active infinitely often}, if along $r$ the number of positions in which player $i$ is the active player is infinite:
\[\big|\{n\in \dZ_0\colon \iota(r_{<n})=i\}\big|\ =\ \infty.\]
We will be mostly interested in games in which along each run there are at least two players who are active infinitely often.

Part 1 of the following lemma shows that if a given player is active infinitely often along each run, then for any (possibly empty) collection of segments, she has a strategy that permits exactly these segments. Part 2 shows that strategy profiles can single out each run: this run is the unique run consistent with the strategy profile. This lemma will be essential when we investigate the relation of our solution concepts (cf.~Theorem \ref{theorem-weak=strong}).

\begin{lemma}[\textbf{Strategies permitting segments and consistent runs}]\label{lem-whichseg} Let $\{\Omega_z\colon z\in Z\}$ be a collection of segments and let $\{r_z\colon z\in Z\}$ be a collection of runs such that $r_z\in \Omega_z$ for all $z\in Z$; these collection can be empty ($Z=\emptyset$ is allowed).
\begin{enumerate}
    \item[{[1]}] Each player $i\in I$ has a strategy $s_i$ such that:
    \begin{itemize}
        \item[{[1.i]}] For all $z\in Z$, the run $r_z$ is consistent with $s_i$.
        \item[{[1.ii]}] Let $r\in R$ be a run such that along $r$ player $i$ is active infinitely often and $r$ does not belong to any of the segments $\Omega_z$, $z\in Z$. Then, $r$ is not consistent with $s_i$.
    \end{itemize} 
    Consequently, if along each run player $i$ is active infinitely often, then $s_i$ permits a segment $\Omega$ if and only if $\Omega=\Omega_z$ for some $z\in Z$.
    \item[{[2]}] Let $r\in R$ be a run. Denote by $\Omega$ the segment of the run $r$. For each player $i\in I$, define the strategy $s_i$ as in Part 1 for the singleton collections $\{\Omega\}$ and $\{r\}$. Then, for the strategy profile $s=(s_i)_{i\in I}$, the run $r$ is the unique run that is consistent with $s$, i.e., $R(s)=\{r\}$. 
\end{enumerate}
\end{lemma}

\begin{proof}\ \\
\noindent\textbf{Proof of Part 1:} 
Let 0 and 1 denote two distinct actions in $A$. First, we define a strategy $x_i$ for player $i$ similar to the one in Example \ref{ex-norun}. For any position $p_n\in P_n$ in any stage $n\in\dZ_0$ where $\iota(p_n)=i$:
\begin{itemize}
    \item if player $i$ always selected action 1 along $p_n$, whenever she was active, then let $x_i(p_n)=0$,
    \item otherwise, let $x_i(p_n)=1$.
\end{itemize}
Note that, as in Example \ref{ex-norun}, if along $r$ player $i$ is active infinitely often, then $r$ is not consistent with $x_i$.

Now we define a strategy $s_i$ for player $i\in I$. For any position $p_n\in P_n$ in any stage $n\in\dZ_0$ where $\iota(p_n)=i$: 
\begin{itemize}
    \item \textsc{Case 1:} Assume that $p_n=r_{z,<n}$ for some $z\in Z$; that is, $p_n$ coincides with the position along the run $r_z$ in stage $n$. Note that, in this case, there is a unique $z\in Z$ with this property, and $p_n\in P_{\Omega_z}$. In this case, define $s_i(p_n)$ to be the action along $r_z$ in stage $n$.
    \item \textsc{Case 2:} Assume the opposite: $p_n\neq r_{z,<n}$ for all $z\in Z$. In this case, define $s_i(p_n)=x_i(p_n)$.
\end{itemize}
By construction, the statement [1.i] holds. To show [1.ii], consider a run $r$ as in [1.ii]. Then, each position along $r$ where player $i$ is active belongs to Case 2, and then according to $s_i$, player $i$ plays according to the strategy $x_i$. Because player $i$ is active infinitely often along $r$, it follows that $r$ is not consistent with $s_i$, proving [1.ii].\medskip\\
\noindent\textbf{Proof of Part 2:} It follows from Part 1 that (i) the run $r$ is consistent with $s$ and 
(ii) because along each run there is a player 
$i$
who is active infinitely often, 
$\Omega$ is the only segment permitted by $s_i$.
By Part 
1 of Lemma \ref{lem-strpercons}, the proof is complete.
\end{proof}

\begin{remark}[\textbf{Permitted segments under deviations}]\label{rem-furPart2}\rm 
Consider again the strategy profile $s$ in Part 2 of Lemma \ref{lem-whichseg}. Assume in addition that, along each run, there are at least two players who are active infinitely often. Then, even if a player $i\in I$ deviates from $s_i$ to some other strategy $s'_i$, under the strategy profile $(s'_i,s_{-i})$ no consistent run can arise in a segment other than $\Omega$. Indeed, this follows from [1.ii] of Lemma \ref{lem-whichseg}. $\blacklozenge$
\end{remark}

\begin{remark}[\textbf{A continuum of consistent runs}]\label{rem-contcons}\rm 
It follows from the previous results that there is a 1-player game in which the player has a strategy that admits a continuum of consistent runs. Indeed, let $Z$ be maximal in Lemma \ref{lem-whichseg}, so that the collection $\{r_z\colon z\in Z\}$ contains exactly one run from each segment. Let $s_1$ be the strategy as in Part 1 of Lemma \ref{lem-whichseg}, so that each $r_z$ is consistent with $s_1$. As the cardinality of $Z$ is equal to the cardinality of the set of all segments, 
which is further equal to the 
cardinality of the 
continuum, $s_1$ indeed admits a continuum of consistent runs. $\blacklozenge$
\end{remark}

\begin{remark}[\textbf{Permitted segments by strategy profiles}]\rm 
Lemma~\ref{lem-whichseg} implies that it might happen that for each player $i$, the strategy $s_i$ permits some segment,
but the strategy profile $s = (s_i)_{i \in I}$ does not permit any segment.
This happens when there is no single segment that is permitted by all strategies $(s_i)_{i \in I}$. $\blacklozenge$
\end{remark}

Based on Lemma \ref{lem-strpercons}, we can define a game on each segment separately with properties that are very similar to standard games.

\begin{definition}[\textbf{Game with infinite past restricted to a segment}]\label{def-gamerestr}\rm
Let $G$ be a game with infinite past and let $\Omega$ be one of its segments. The \emph{game $G$ restricted to $\Omega$, denoted by $G_\Omega$}, is given as in Definition \ref{def-game}, with the following modifications: (1) the set of positions is $P_\Omega$, (2) the set of runs is $\Omega$, and (3) the set of strategies of each player $i\in I$ is $S_i(\Omega)$: the set of strategies that permit $\Omega$. Consequently, the set of strategy profiles is $S(\Omega):=\prod_{i\in I}S_i(\Omega)$.
\end{definition}

By part 2 of Lemma \ref{lem-strpercons},
in the game $G_\Omega$, 
for any strategy profile there is a unique consistent run,
and hence
we can define payoffs and 
equilibrium in the usual manner: In $G_\Omega$, define the payoff of each player $i\in I$ under each strategy profile $s\in S(\Omega)$ as
\begin{equation}\label{eq-payoffomega}
u_{i,\Omega}(s)\,:=\,u_i(r),
\end{equation}
where $r$ is the unique run in $\Omega$ that is consistent with $s$. A strategy profile $s\in S(\Omega)$ is a (traditional) equilibrium in $G_\Omega$ if $u_{i,\Omega}(s'_i,s_{-i})\leq u_{i,\Omega}(s)$ for each player $i\in I$ and each strategy $s'_i\in S_i(\Omega)$.

\section{Equilibrium}\label{sec-eq}

In this section we define and examine two equilibrium concepts in games with infinite past. While these two concepts are related, there is a key difference between them in how runs that are consistent with strategy profiles are treated. 

The section is structured as follows. In Section \ref{sec-defeq}, we define these two equilibrium concepts. In Section \ref{sec-examples}, we provide 
illustrating
examples. 
In Section \ref{sec-releq}, we compare these concepts in detail, and prove the main result of this section: in every game that satisfies the condition that along each run at least two players are active infinitely often, there is no difference between the two concepts as far as the existence of an equilibrium and the set of equilibrium runs are concerned. 
Finally, in Section \ref{sec-backind}, we discuss the method of backward induction in games with infinite past to determine an equilibrium.

\subsection{Two equilibrium concepts}\label{sec-defeq}

Our main equilibrium concept is a pair, consisting of a strategy profile and a run that is consistent with this strategy profile, satisfying the following natural requirement: there is no position along this run after which a player has a profitable deviation from her own strategy. 

Recall that each subgame $G[p]$, where $p\in P$ is a position, is a finite game, and hence in $G[p]$ the traditional concept of equilibrium can be defined in the usual way. Note that each strategy profile $s$ induces a strategy profile $s[p]$ in the subgame $G[p]$. 
To distinguish between the concept of equilibrium in subgames and the equilibrium concepts in games with infinite past, 
we call the former \emph{traditional} equilibrium.
 
\begin{definition}[\textbf{Equilibrium}]\label{def-weakeq}\rm 
Let $s=(s_i)_{i\in I}$ be a strategy profile and let $r\in R(s)$ be a run consistent with $s$. We call the pair $(s,r)$ an \emph{equilibrium} if, for each stage $n\in\dZ_0$, the strategy profile $s[r_{< n}]$ is a traditional equilibrium in the subgame $G[r_{< n}]$. A run $r\in R$ is called an \emph{equilibrium run} if there is a strategy profile $s$ such that $(s,r)$ is an equilibrium.
\end{definition}

As each subgame $G[p]$ is a finite game, it admits a traditional equilibrium. But the concept of an equilibrium requires more: the same strategy profile $s$ should induce a traditional equilibrium in all subgames along the run $r$, and moreover, the run $r$ should be consistent with $s$.

Thus, the concept of an equilibrium only takes into account deviations in a finite number of stages, i.e., only looks at subgames that belong to the segment $\Omega^r$ of the run~$r$. Note that $(s,r)$ is an equilibrium if and only if (i) $s$ is an equilibrium in the game $G_{\Omega^r}$ (cf.~Definition \ref{def-gamerestr} and \Eqref{eq-payoffomega}), and (ii) $r$ is the unique run consistent with $s$ in $\Omega^r$ (cf.~Part 2 of Lemma~\ref{lem-strpercons}).

\begin{remark}[\textbf{Games with tail payoff functions}]\label{rem-tail} \rm Tail payoffs have been studied recently in various classes of dynamic games 
(see for instance Ashkenazi-Golan et al. [2022] and Flesch and Solan [2023], and the references therein).
The payoff function $u_i$ of some player $i\in I$ is called \emph{tail} if changing finitely many actions in the run does not affect the payoff, 
that is, if $u_i$ is a constant on each segment. 
For example, the payoff function of player 1 in Example \ref {ex-eqnostrong} 
below 
is tail. 

It follows\footnote{This statement is driven by the fact that, after each stage, the set of remaining stages is finite. This statement would no longer be true if the set of stages were $\mathbb{Z}$, the set of all integers.} from Definition \ref{def-weakeq} that if each player's payoff function is tail, 
then $(s,r)$ is an equilibrium, for any strategy profile $s$ and any consistent run $r\in R(s)$. $\blacklozenge$
\end{remark}

The second equilibrium notion, called strong equilibrium, consists only of a strategy profile that admits a unique consistent run. Here the equilibrium property requires that, for any player, her payoff under this run be at least as much as her payoff under any run that is consistent if she  deviates from her strategy. Thus, strong equilibrium considers all segments of the game when a deviation takes place.

\begin{definition}[\textbf{Strong Equilibrium}]\label{def-strongeq}\rm 
Let $s=(s_i)_{i\in I}$ be a strategy profile that admits a unique consistent run $r\in R$, i.e., $R(s)=\{r\}$. We call $s$ a \emph{strong equilibrium} if for each player $i\in I$ and each strategy $s'_i$ of player $i$ such that $R(s'_i,s_{-i})\neq \emptyset$ we have 
    \begin{equation}\label{eq-strong}
    u_i(r')\,\leq\,u_i(r),\hspace{0.75cm}\forall r'\in R(s'_i,s_{-i}).
    \end{equation}
A run $r\in R$ is called a \emph{strong equilibrium run} if there is a strong equilibrium $s$ with $R(s)=\{r\}$.
\end{definition}

The following theorem follows by the definitions.

\begin{theorem}[\textbf{Strong equilibrium implies equilibrium}]\label{theorem-weakstrong}
Assume that $s$ is a strong equilibrium. Let $r$ be the unique run consistent with $s$, that is, $R(s)=\{r\}$. Then $(s,r)$ is an equilibrium.
\end{theorem}

\subsection{Examples}
\label{sec-examples}

In this section we provide several examples that illustrate the two concepts of equilibria and some of their properties.
Example~\ref{ex-noeq} provides a
1-player game that admits no equilibrium, and therefore,
by Theorem~\ref{theorem-weakstrong}, 
no strong equilibrium either.
This property is not surprisingly, 
since the payoff functions are not required to be continuous.
Example~\ref{ex-eqnostrong} provides a
1-player game with equilibrium but no strong equilibrium,
thereby showing that the two concepts of equilibria are distinct.
Finally, Example~\ref{ex-winloseruns} provides a zero-sum game 
without a value: 
each segment separately has a value, but different segments have a different value. This is probably not unexpected,
as distinct segments do not have a common position and
the payoff function restricted to each of them
can be very different. Still, this example 
illustrates that properties that hold in normal-form games and repeated games (with infinite future) need not hold for games with infinite past. 

\begin{example}[\textbf{1-player game without equilibrium}]\label{ex-noeq}\rm 
Consider the following 1-player game. The action set is $A=\{0,1\}$. Let $\delta\in(0,1)$. Along a run $r=(\ldots,a_{-1},a_0)$, if the player always chose action 1 then the payoff is 0, and otherwise the payoff is equal to the reversed-time discounted sum
\[(1-\delta)\cdot \sum_{n=-\infty}^0 \delta^{-n}a_n\,=\,(1-\delta)\cdot(a_0+\delta a_{-1}+\delta^2 a_{-2}+\cdots);\]
note that the latter quantity always belongs to the interval $[0,1)$. 

This game admits no equilibrium. Indeed, (i) if $r$ only contains action 1, then it is a profitable deviation to play action 0 in stage 0, (ii) if $r$ contains action 0 in only one stage, say stage $n$, then it is a profitable deviation to play action 0 in stage $n-1$ and play action 1 in stage $n$, and (iii) if $r$ contains action 0 in at least 2 stages, then it is a profitable deviation to play action 1 in one of those stages. $\blacklozenge$
\end{example}

\begin{example}[\textbf{1-player game with equilibrium but no strong equilibrium}]\label{ex-eqnostrong}\rm 
Consider the following 1-player game. The action set is 
$A=\{0,1\}$. Along a run $r=(\ldots,a_{-1},a_0)$, let the limsup-frequency of action 1 be denoted by
\[\phi(r)\,=\,\limsup_{n\to-\infty}\frac{1}{-n+1}(a_{n}+\cdots+a_{-1}+a_0).\]
The payoff for a run $r\in R$ is as follows: $u_1(r)=\phi(r)$ if  $\phi(r)<1$, and $u_1(r)=0$ if  $\phi(r)=1$. 

In this game, $(s,r)$ is an equilibrium for any strategy $s$ and consistent run $r$. Indeed, given any stage $n\in\dZ_0$, the player is unable to change the limsup-frequency of action 1 by deviating only after stage $n$ (cf.~Remark \ref{rem-tail}).

However, there is no strong equilibrium. Indeed, let $s$ be any strategy and $r\in R(s)$. Then, $u_1(r)<1$, and hence there is a run $\overline r\in R$ such that $u_1(\overline r)>u_1(r)$.
By Part 2 of Lemma \ref{lem-whichseg}, there is a strategy $\overline{s}$ such that $R(\overline s)=\{\overline r\}$. 
This means that $s$ is not a strong equilibrium. $\blacklozenge$
\end{example}

\begin{remark}[\textbf{$\ep$-equilibria}]\rm 
As is common when the payoff function is not continuous,
one could define the concepts of $\ep$-equilibria and strong $\ep$-equilibria
in analogy to those of equilibria and strong equilibria.
In one-player games with infinite past,
a strong $\ep$-equilibrium always exists.
Indeed, any run that maximizes the payoff up to $\ep$
induces a strong $\ep$-equilibrium,
and hence also an $\ep$-equilibrium.
As we will see below, when at least two players participate in the game,
$\ep$-equilibria may fail to exist when $\ep$ is sufficiently small. $\blacklozenge$
\end{remark}

\begin{example}[\textbf{A zero-sum game without a value}]\label{ex-winloseruns}\rm 
We consider a two-player zero-sum game $G$ with infinite past in which player 1 is the active player in even stages and player 2 is active in odd stages. The action set is $A=\{0,1\}$. Let $W$ denote the set of runs in which player 2 played action~1 at least once. The payoff is $u_1(r)=1$ and $u_2(r)=-1$ if $r\in W$, and $u_1(r)=-1$ and $u_2(r)=1$ if $r\notin W$. Note that the payoffs are independent of player 1's actions.%
\footnote{The game is in fact a win-lose game, and such games will be studied in Section \ref{sec-winlose}.}

Let $s^0$ (resp., $s^1$) be the strategy profile in which each player always plays action 0 (resp., $1$), and let the run $r^0$ (resp., $r^1$) consist of action 0 only (resp., $1$). Then, both $(s^0,r^0)$ and $(s^1,r^1)$ are equilibria with $r^0\notin W$ and $r^1\in W$. 
While the first equilibrium is strong, the second is not. 
Below (Theorem~\ref{theorem-weak=strong})
we will show that in fact there is a strong equilibrium $\widehat s^1$ such that $r^1$ is the unique run consistent with $\widehat s^1$.
It follows that the zero-sum game admits no value, and hence, no strong value either.

Yet, each segment separately does have a value:\footnote{For a precise discussion of the concept of the value in a segment, we refer to Section \ref{sec-segwinlose}.} for each segment $\Omega$, the value of the game $G_\Omega$ is equal to 1 or $-1$. Indeed, in each segment $\Omega$, either along every run player 2 played action 1 only finitely often (such as in the segment of $r^0$), or along every run player 2 played action 1 infinitely often (such as in the segment of $r^1$). In the former case, the value of $G_\Omega$ is $-1$, as player 2 can always play action 0. In the latter case, the value of $G_\Omega$ is 1, as every run belongs to $W$. 
$\blacklozenge$
\end{example}

\subsection{Relation between the two equilibrium concepts}\label{sec-releq}

In this section we compare the concepts of equilibrium and strong equilibrium. 

The next theorem shows a converse of Theorem \ref{theorem-weakstrong}. It states that 
the existence of an equilibrium implies the existence of a strong equilibrium with the same equilibrium run, in related strategies, under the mild condition that along each run at least two players are active infinitely often. This condition is essential: recall from  Example \ref{ex-eqnostrong} that when along some runs only one player is active infinitely often,
the game may admit an equilibrium but no strong equilibrium.

\begin{theorem}[\textbf{Equilibrium implies strong equilibrium}]\label{theorem-weak=strong}
Assume that along each run at least two players are active infinitely often.

Let $(s,r)$ be an equilibrium. Let $\widehat s$ be the strategy profile as in Part 2 of Lemma \ref{lem-whichseg} for the run $r$. Further, let $s^*$ be the strategy profile that coincides with $\widehat s$, except if the play follows $r$ until some stage $n$, where a player deviates from $r$ (or, equivalently, from $\widehat s$). In that case, $s^*$ switches to $s$ for the remaining stages $n+1,\ldots,-1,0$.  

Then, $s^*$ is a strong equilibrium and $r$ is the corresponding equilibrium run.
\end{theorem}

\begin{proof}
By the choice of $\widehat s$, we have $R(\widehat s)=\{r\}$. 
Since $s^*$ follows $\widehat s$ unless 
a deviation occurs from $r$ (and thus from $\widehat s$), we also have $R(s^*)=\{r\}$.

It remains to show that $s^*$ is a strong equilibrium. So, suppose that a player $i\in I$ deviates from $s^*_i$ to some strategy $s'_i$. Let $r'$ be a run consistent with $(s'_i,s^*_{-i})$, that is, $r'\in R(s'_i,s^*_{-i})$, and suppose that $r'\neq r$. We need to show that $u_i(r')\leq u_i(r)$.

Let $\Omega$ denote the segment of $r$. Because by assumption, along each run at least two players are active infinitely often, by Remark \ref{rem-furPart2}, no run in $R\setminus \Omega$ is consistent with $(s'_i,\widehat s_{-i})$. Because $s^*$ coincides with $\widehat s$ along the runs in $R\setminus \Omega$, no run in $R\setminus \Omega$ is consistent with $(s'_i,s^*_{-i})$. Thus, the runs $r$ and $r'$ both belong to the same segment $\Omega$. 

Let $n$ denote the earliest stage in which the actions in $r$ and $r'$ differ. Then, player $i$ is the active player in the position $r_{<n}=r'_{<n}$. Moreover, along $r'$, after player $i$'s action in stage $n$, the strategy profile $s^*_{-i}$ coincides with $s_{-i}$ in stages $n+1,\ldots,-1,0$. Because $(s,r)$ is an equilibrium, we have $u_i(r')\leq u_i(r)$.
\end{proof}

The following result follows from Theorems \ref{theorem-weakstrong} and \ref{theorem-weak=strong}.

\begin{theorem}[\textbf{Equilibrium runs coincide with strong equilibrium runs}]\label{theorem-runsequal}
Assume that along each run at least two players are active infinitely often. Then, the set of equilibrium runs coincides with the set of strong equilibrium runs. 
\end{theorem}

\subsection{Backward induction and 
games with infinite past
}\label{sec-backind}

In this section, first we examine to which extent backward induction can be used in games with infinite past to find an equilibrium. Then, we explain our proof techniques, and compare them with backward induction.\medskip

In standard perfect information games with finite horizon, backward induction is known to provide a traditional (subgame perfect) equilibrium. 
One may wonder to which extent backward induction can be used in games with infinite past to determine an equilibrium. Note that, in a game with infinite past, the set of stages is infinite, and hence backward induction needs to be executed in an infinite number of iteration steps in the stages $0,-1,-2,\ldots$. 
Below we will argue 
that 
this causes quite a few difficulties.

What is always true is that the backward induction procedure ends up with a strategy profile, say $s$, such that $s$ induces a traditional equilibrium in each subgame. And whenever this strategy profile $s$ admits a consistent run $r$, then $(s,r)$ is an equilibrium according to Definition \ref{def-weakeq}.

But as we have already seen, strategy profiles do not always admit a consistent run, and this is also true for strategy profiles that are determined by backward induction. Indeed, for instance, the strategy in Example \ref{ex-norun} is a backward induction outcome when the player's payoff function is constant. Or, consider Example~\ref{ex-noeq}. In this game, backward induction leads to a unique strategy, which turns out to be identical to the strategy in Example~\ref{ex-norun}: in a position $p$, play action 0 if $p$ only contains action 1, and play action 1 otherwise. Hence, backward induction does not lead to a consistent run.

Since the existence of a consistent run is a main issue, let us approach this question closer from the viewpoint of runs. For each stage $n\in\dZ_0$, let $B_n$ denote the set of runs $r=(\ldots,a_{-1},a_0)$ such that $(a_n,\ldots,a_{-1},a_0)$ is a backward induction play in the subgame starting at the position $r_{<n}=(\ldots,a_{n-2},a_{n-1})$. Note that each $B_n$ is nonempty, and the sequence $B_n$ is nested: $B_0\supseteq B_1\supseteq \cdots$. If the sets $B_n$ are also closed (and hence compact), then Cantor's intersection theorem\footnote{Cantor's intersection theorem states that every nonincreasing nested sequence of nonempty compact subsets of a Hausdorff topological space has a nonempty intersection.} implies that the intersection $B^*=\bigcap_{n=0}^{-\infty}B_n$ is nonempty, and it follows by construction that each element of $B^*$ is an equilibrium run. This idea however fully hinges upon on the closedness of the sets $B_n$.

To illustrate this, consider again Example \ref{ex-noeq}, where backward induction delivers a unique action in each position. In this example, 
$B_n$ contains the following runs:
\begin{itemize}
    \item the runs $(\ldots,a_{n-2},a_{n-1},1,1,\ldots,1)$ where $(\ldots,a_{n-2},a_{n-1})\neq(1\ldots,1,1)$,
    \item the run $(\ldots,1,1,0,1,\ldots,1)$ that contains action 0 only in stage $n$. 
\end{itemize}
In this case, the sets $B_n$ are not closed. Indeed, for each $m<n$, the run $r^m$ that contains action 0 only in stage $m$ belongs to $B_n$, but the limit of $r^m$ as $m\to-\infty$, the run that only contains action 1, does not belong to $B_n$.\medskip

\noindent\textbf{Our proof technique.} In the proofs of our existence results, we will regularly use Cantor’s intersection theorem, albeit somewhat differently. Let us denote by $E_n$ the set of runs $r=(\ldots,a_{-1},a_0)$ such that $(a_n,\ldots,a_{-1},a_0)$ coincides with an equilibrium play in the subgame starting at the position $r_{<n}=(\ldots,a_{n-2},a_{n-1})$. Since we are only interested in equilibria, we do not need to require that $(a_n,\ldots,a_{-1},a_0)$ is a backward induction play; thus $B_n\subseteq E_n$. 
Then, for each $n\in\dZ_0$, we define a subset $R_n\subseteq E_n$ such that, under the assumptions of the respective existence result, $R_n$ is nonempty and closed, and the sequence $R_n$ is nested: $R_0\supseteq R_1\supseteq \cdots$. This makes sure that, by Cantor’s intersection theorem, the intersection $R^*=\bigcap_{n=0}^{-\infty}R_n$ is nonempty and that each run $r\in R^*$ can be supplemented with a strategy profile $s$ such that $(s,r)$ is an equilibrium. 

\section{Two-player win-lose games}\label{sec-winlose}

As mentioned in the introduction,
two-player win-lose games,  
are a specific class of zero-sum games, 
which are frequently studied in the game-theoretic literature, and they play a major role in the field of descriptive set theory. In this section we examine these games in the presence of infinite past. 

In a two-player win-lose game, there are only two outcomes: either player 1 wins and player 2 loses, or player 1 loses and player 2 wins. 
In such a game with infinite past, we can 
divide the set of runs into 
two: we let $W$ denote the set of all runs where player~$1$ wins
(so that player~$2$ wins at all runs in $R \setminus W$).

In standard games where the set of stages is the set $\dN$ of natural numbers, a landmark result of Martin [1975] showed that an equilibrium exists\footnote{Or, in terms of zero-sum games, one of the players has a winning strategy.} whenever the winning set $W$ is Borel-measurable. In this section we show a sharp contrast for games with infinite past:
\begin{itemize}
    \item An equilibrium exists as long as the winning set $W$ has Borel-rank at most 2 (cf.~Theorem \ref{theorem-rank2}). Our proof relies on a result when the winning set $W$ has Borel-rank 1 (cf.~Lemma \ref{theorem-open}), where we prove some useful properties for equilibria beside their existence.
    \item An equilibrium does not always exist when the winning set $W$ has Borel-rank 3 or higher (cf.~Example \ref{ex-counter}). The reason is, intuitively, as follows: Suppose that each player prefers to play a specific action $a^*$, but she only wants to play it finitely many times. 
    If both players play $a^*$ finitely many times, then the winner is the player who started to play $a^*$ earlier than the other player. This type of constructions yield winning sets of Borel-rank 3, but they lead to no equilibrium.\footnote{In standard games where the set of stages is the set $\dN$ of natural numbers, the following fairly general statement holds: if each player's payoff function has finite range, regardless of whether the payoff functions are Borel-measurable, there always exists a strategy profile satisfying the one-deviation principle, cf.~Flesch et al. [2010]. Example \ref{ex-counter} thus also shows that a similar results does not hold in games with infinite past.}
\end{itemize}

The section is structured as follows.
First, we define two-player win-lose games and state some of their basic properties. 
Then, in Section \ref{sec-rev}, we define reversed positions and their relations to open winning sets. In Section \ref{sec-open} we prove an auxiliary result for two-player win-lose games with an open winning set (cf.~Lemma \ref{theorem-open}). In Sections \ref{sec-rank2} and \ref{sec-rank3} we state and prove 
the main results of this section (cf.~Theorems \ref{theorem-rank2} and Example \ref{ex-counter}).

\begin{definition}[\textbf{Two-player win-lose game with infinite past}]\label{ded-winlose}\rm 
\label{definition:win lose game}
A game with infinite past is called a \emph{two-player win-lose game} if:
\begin{itemize}
    \item[(1)] There are two players who move alternately; that is, $I=\{1,2\}$, player 1 is active in even stages, and player 2 is active in odd stages.
    \item[(2)] There is a set $W\subseteq R$ such that for each run $r\in R$: if $r\in W$ then $u_1(r)=1$ and $u_2(r)=0$, while if $r\in R\setminus W$ then $u_1(r)=0$ and $u_2(r)=1$. 
    The set $W$ is called the \emph{winning set} of player 1 and the set $R\setminus W$ is called the winning set of player 2.
    It is also common to call $W$ the winning set of the game.
\end{itemize}
\end{definition}

\begin{remark}[\textbf{On the assumption of alternating moves}]\rm 
Following the literature of two-player win-lose games (cf.~Kechris [2012]), we assume in Definition \ref{ded-winlose} that the players move alternately. In particular, this implies that each player is active infinitely often along each run.%
\footnote{And vice versa, as long as each player is active infinitely often along each run, it is without loss of generality to assume that the players play alternately. 
Indeed, to ensure that this condition holds, one can add dummy stages  
in which the player's action choices do not affect the outcome, and the addition of these dummy stages does not affect the rank of the winning set in the Borel hierarchy, because an open winning set remains open after this operation.} 
Hence, we can make use of Theorem \ref{theorem-runsequal} showing that equilibrium runs and strong equilibrium runs coincide in these games.
In the remainder of the section, we focus on the existence of equilibrium and equilibrium runs.
\end{remark}




\begin{definition}[\textbf{Winning position}]\label{def-winningpos}\rm 
Consider a two-player win-lose game, and a position $p_n$ in some stage $n\in\dZ_0$. We call the position $p_n$ \emph{winning} for player 1 if, in the subgame $G[p_n]$, player 1 has a winning strategy: by playing this strategy in\footnote{Recall that player 1 is only active in the even stages of the subgame.} stages $n,\ldots,-1,0$, player 1 ensures that the induced run (which starts with $p_n$) belongs to $W$, regardless of the strategy of player 2. A winning position for player 2 is defined similarly.
\end{definition}

In the next lemma, Part~1 follows from the fact that each subgame is a finite game, and Part~2 follows from the definition of winning positions.
\color{black}

\begin{lemma}[\textbf{Properties of winning positions}]
\label{lem-winpos}\ 
\begin{itemize}
    \item[(1)] Every position is winning for exactly one of the players. 
    \item[(2)] Assume that a position $p_n\in P_n$ in stage $n<0$ is winning for player $i\in\{1,2\}$. If player~$i$ is the active player at $p_n$, 
    then player $i$ has an action $a_n\in A$ such that by playing this action the next position $(p_n,a_n)$ is also winning for player $i$. If 
    player $j\neq i$
    is the active player at $p_n$, 
    then regardless of which action $a_n\in A$ player $j$ plays, the next position $(p_n,a_n)$ is still winning for player $i$.
    The same statements hold for every position $p_0\in P_0$, the action $a_0$, and the resulting run $(p_0,a_0)$.
\end{itemize}
\end{lemma}

The next lemma relates equilibrium to winning positions.

\begin{lemma}[\textbf{Equilibrium and winning positions}]
\label{lemma-howwins}\ 
\begin{itemize}
    \item[(1)] Assume that $(s,r)$ is an equilibrium. If $r\in W$, then each position along $r$, i.e., $r_{<n}$ for all $n\in\dZ_0$, is winning for player 1. Similarly, if $r\notin W$, then each position along $r$ is winning for player 2.
    \item[(2)] If $r\in W$ and player 1 is winning in each position along the run $r$, 
    then there is a strategy profile $s$ such that $(s,r)$ is an equilibrium. Similarly, if player 2 is winning in each position along $r$, then there is a strategy profile $s$ such that $(s,r)$ is an equilibrium.\footnote{In the first claim of Part 2, we need to assume that $r\in W$. Indeed, because a run is not a position, the fact that the position in stage 0 is winning for player 1 does not imply that player 1 plays a winning action $a_0$ along the run $r$. In the second claim of Part 2, it always holds that $r\notin W$. Indeed, because the position in stage 0 is winning for player 2, no matter the action $a_0$ of player 1 in stage 0, the run $r$ is winning for player~2.}
\end{itemize}
\end{lemma}

\begin{proof}
Part 1 follows from the definitions. 

We next prove the first statement in Part 2; the proof of the second statement is similar. Assume that $r\in W$ and player 1 is winning in each position along $r$. Let the strategy profile $s$ be as follows: (1) $s$ recommends the action corresponding to $r$ in each position along $r$, and (2) as soon as player 2 deviates from $r$, player 1 continues with a winning strategy in the remaining subgame. 
In runs that are not in the segment that contains $r$,
the strategy profile $s$ is defined arbitrarily. 
Part 2 of Lemma \ref{lem-winpos} ensures that such a winning strategy indeed exists. Then, $(s,r)$ is an equilibrium.
\end{proof}

Lemma \ref{lemma-howwins} identifies the key step, and highlights the main difficulty, in constructing an equilibrium in two-player win-lose games: we need to find a run along which the same player is winning in all positions. 

\subsection{Reversed positions and open winning sets}\label{sec-rev}

We define reversed positions, which are strongly related to the product topology of the set of runs and simplify the discussion of subgames.

\begin{definition}[\textbf{Reversed position}]
A \emph{reversed position} in stage $n\in\dZ_0$ is a finite sequence $q_n=(a_n,\ldots,a_{-1},a_0)$ of actions. 
\end{definition}
\color{black}

Let $Q_n$ denote the set of all reversed positions in stage $n$, and let $Q=\bigcup_{n\in\dZ_0}Q_n$ be the set of all reversed positions. For a reversed position $q_n\in Q_n$, let $[q_n]$ denote the set of runs $r\in R$ that end with the sequence $q_n$ in stages $n,\ldots,-1,0$:
\[[q_n]\,:=\Big\{r=(\ldots,a_{-1},a_0)\in R\colon (a_n,\ldots,a_{-1},a_0)=q_n\Big\};\] such a set $[q_n]$ is often called a \emph{cylinder set}.

Suppose that $W\subseteq R$ is an open set of runs in the product topology. Then, $W$ is a union of cylinder sets.
That is,
\begin{equation}
\label{eq-Wunion}
W = \bigl\{ [q] \colon q \in \widehat Q\bigr\},
\end{equation}
for some set $\widehat Q\subseteq Q$ of reversed positions.%
\footnote{Note that generally the set $\widehat Q$ is not uniquely defined.}
For  each $n\in\dZ_0$, set $\widehat Q_n = \widehat Q \cap Q_n$ be the set of reversed positions of $\widehat Q$ in stage $n$.
Thus, a run $r=(\ldots,a_{-1},a_0)\in R$ belongs to $W$ if and only if $r$ ends with a reversed position in $\widehat Q$: for some stage $n\in\dZ_0$ we have $(a_n,\ldots,a_{-1},a_0)\in\widehat Q_{n}$.
For each $n\in\dZ_0$, let 
\begin{equation}\label{eq-whQ_n}
\widehat Q_{\geq n}\,=\,\widehat Q_n\cup\cdots\cup\widehat Q_{-1}\cup\widehat Q_0.
\end{equation}

\begin{definition}[\textbf{The auxiliary game $\widehat G_{\geq n}$}]
Given 
a set of reversed positions $\widehat Q\subseteq Q$
and $n\in\dZ_0$, 
define
an \emph{auxiliary game} $\widehat G_{\geq n}$ as follows. For each $n\in\dZ_0$, let 
$\widehat G_{\geq n}$ be the game played in stages $n,\ldots,-1,0$ in which:
\begin{itemize}
\item Player 1 is active in even stages and player 2
is active in odd stages (like in the original two-player win-lose game $G$).
\item Player 1 wins along $(a_n,\ldots,a_{-1},a_0)$ if it ends with a reversed position in $\widehat Q_{\geq n}$: for some $k\in\{n,\ldots,-1,0\}$ we have $(a_k,\ldots,a_{-1},a_0)\in\widehat Q_{k}$; and player 2 wins otherwise.
\end{itemize}
\end{definition}
\color{black}

We emphasize that $\widehat G_{\geq n}$ is a finite duration game. 
Therefore, one of the players has a winning strategy in $\widehat G_{\geq n}$: playing this strategy ensures a win in $\widehat G_{\geq n}$ for this player, regardless of the strategy of the opponent.

The sequence of sets $\widehat Q_{\geq 0},\widehat Q_{\geq -1},\widehat Q_{\geq -2},\ldots$ is increasing. Thus:
\begin{itemize}
    \item If player~$1$ has a winning strategy in $\widehat G_{\geq n}$,
then she also has a winning strategy in $\widehat G_{\geq n-1}$. Indeed, after any action $a_{n-1}$, the winning strategy of player 1 in $\widehat G_{\geq n}$ makes sure in the remaining stages $n,\ldots,-1,0$ that the run ends with a reversed position in $\widehat Q_{\geq n}$. And as $\widehat Q_{\geq n}\subseteq \widehat Q_{\geq n-1}$, 
the run ends with a reversed position in $\widehat Q_{\geq n-1}$.
\item Similarly, if player~$2$ has a winning strategy in $\widehat G_{\geq n-1}$,
then she also has a winning strategy in $\widehat G_{\geq n}$. Indeed, the winning strategy of player 2 in $\widehat G_{\geq n-1}$ makes sure in stages $n-1,n,\ldots,-1,0$ that the run does not end with a reversed position in $\widehat Q_{\geq n-1}$. And as $\widehat Q_{\geq n}\subseteq \widehat Q_{\geq n-1}$, 
the run does not end with a reversed position in $\widehat Q_{\geq n}$ either.
\end{itemize}

\subsection{An auxiliary result for open winning sets}\label{sec-open}

The following lemma separates two cases for open winning sets in two-player win-lose games, and shows that an equilibrium exists in both cases. The proof of the second case is the more difficult one, and here we will use Cantor's intersection theorem as explained in Section \ref{sec-backind}. We will use this lemma to prove in Theorem \ref{theorem-rank2} that all two-player win-lose games with winning sets of Borel-rank at most 2 admit an equilibrium.

\begin{lemma}[\textbf{Implications of having a winning strategy in $\widehat G_{\geq n}$ when $W$ is open}]
\label{theorem-open}
Consider a two-player win-lose game $G$. Assume that player 1's winning set $W$ is open in the product topology of $R$, and let $\widehat Q$ be a set of reversed positions such that $W$ can be written  as in \Eqref{eq-Wunion}. 
We distinguish two cases:
\begin{itemize}
    \item[(1)] Assume that player~1 has a winning strategy in the game $\widehat G_{\geq n}$ for some $n\in\dZ_0$. Then, in the game $G$:
    \begin{itemize}
        \item[(1-i)] Each position in stage $n$ is winning for player 1.
        \item[(1-ii)] There exists an equilibrium $(s,r)$ with $r\in W$.
        \item[(1-iii)] For every equilibrium $(s,r)$ we have $r\in W$.
    \end{itemize}
    \item[(2)] Assume that player~2 has a winning strategy in the game $\widehat G_{\geq n}$ for every $n\in\dZ_0$. Then, the game $G$ admits an equilibrium $(s,r)$ with $r\notin W$.
\end{itemize}
\end{lemma}

\begin{remark}[\textbf{Dependence of the conclusion of Lemma~\ref{theorem-open} on the choice of $\widehat{Q}$}]\rm 
As mentioned earlier, generally the winning set $W$ 
can be written as in \Eqref{eq-Wunion} with multiple sets $\widehat Q$. 
Depending on which $\widehat Q$ we take, we may obtain different auxiliary games $\widehat G_{\geq n}$. So one may wonder whether the conclusions of Lemma~\ref{theorem-open} depend on the choice of $\widehat Q$. The answer to this question is negative. Due to Parts (1-iii) and (2) of Lemma~\ref{theorem-open}, either all the possible sets $\widehat Q$ lead to Part 1, or all the possible sets $\widehat Q$ lead to Part 2. $\blacklozenge$
\end{remark}

\begin{proof}\ \\
\noindent \textbf{Proof of Part 1:} As player~1 has a winning strategy in the game $\widehat G_{\geq n}$, player 1 can ensure during the stages $n,\ldots,-1,0$ of $\widehat G_{\geq n}$ that the run ends with a reversed position in $\widehat Q_{\geq n}$, which implies by Eqs.~\eqref{eq-Wunion} and \eqref{eq-whQ_n} that the run will belong to $W$. In particular, Part (1-i) holds.

Now we argue that Part (1-ii) holds. Let $s_1$ be a strategy for player 1 in the game $G$ that plays a fixed action $a^*\in A$ before stage $n$ and then ensures in stages $n,\ldots,-1,0$ that the run ends with a reversed position in $\widehat Q_{\geq n}$. Let $s_2$ be the strategy for player 2 that always plays the fixed action $a^*\in A$. Then, there is a unique run $r$ that is consistent with $s=(s_1,s_2)$, and $r\in W$ due to $s_1$. 
Moreover, for every strategy $s'_2$ for player~2
and every run $r'$ that is consistent with $(s_1,s'_2)$,
the run $r'$ must end with a reversed position in $\widehat Q_{\geq n}$.
It follows that $(s,r)$ is an equilibrium.

Finally, Part (1-iii) holds, since the run before stage $n$ does not affect the winning condition in $\widehat G_{\geq n}$,
and since in $\widehat G_{\geq n}$ player~1 has a winning strategy.
\smallskip

\noindent \textbf{Proof of Part 2:}
This is the difficult part of the proof, and we will use Cantor's intersection theorem 
(see Section \ref{sec-backind} for a general description of the idea).

As before, we use the letters $p$ and respectively $p_n$ for positions and positions in stage $n$ in the original game $G$. That means that $p_n=(\ldots,a_{n-2},a_{n-1})$ is an infinite sequence of actions until stage $n-1$. 

Recall that the game $\widehat G_{\geq k}$ is played in stages $k,\ldots,-1,0$. It has a finite set of positions: $A^{\leq -k}=A^0\cup\cdots\cup A^{-k}$, which is all finite sequences of actions of length at most $-k$, including the empty sequence. 

If $k\leq n$, then a position $p_n=(\ldots,a_{n-2},a_{n-1})$ in stage $n$ of the game $G$ naturally induces the position $(a_k,\ldots,a_{n-2},a_{n-1})$ in stage $n$ of the game $\widehat G_{\geq k}$. For simplicity, we will write that $p_n$ is a position in the game $\widehat G_{\geq k}$. 
This will allow us to evaluate how good a position $p_n$ of the game $G$ is with respect to all $\widehat G_{\geq k}$ where $k\leq n$, that is, with respect to all $\widehat Q_{\geq k}$ where $k\leq n$. This is of primary interest to the players as $W$ is determined by the sets $\widehat Q_{\geq k}$, cf.~\Eqref{eq-Wunion} and \eqref{eq-whQ_n}.

Note that two different positions $p_n$ and $p'_n$ in stage $n$ of the game $G$ can induce the same position in a game $\widehat G_{\geq k}$. This happens if $p_n$ and $p'_n$ only differ before stage $k$.\medskip

\noindent\textbf{Step 1: Defining the quantity $w(p_n)$.} For each position $p_n=(\ldots,a_{n-2},a_{n-1})$ in the game $G$, in each stage $n\in\dZ_0$, let 
\begin{equation}\label{eq-w}
w(p_n)\,=\,\inf\Big\{k\in\{\ldots,n-1,n\}\,\colon\,\text{the position $p_n$ is winning for player 2 in }\widehat G_{\geq k}\Big\}.
\end{equation}

Note that $w(p_n)\leq n$. Moreover, if $m\leq\ell\leq n$, then $\widehat Q_{\geq m}\supseteq \widehat Q_{\geq \ell}$, and hence in the position $p_n$, the game $\widehat G_{\geq m}$ is (weakly) favorable for player 1 compared to $\widehat G_{\geq \ell}$, and $\widehat G_{\geq \ell}$ is (weakly) favorable for player 2 compared to $\widehat G_{\geq m}$. Thus, intuitively, the lower $w(p_n)$ is, the more favorable
this position $p_n$ is for player 2, and the ideal situation for player 2 is when $w(p_n) = -\infty$.\medskip

\noindent\textbf{Illustration of the definition of $w(p_n)$:}  
Suppose that the action set is $A=\{0,1\}$ and the open winning set $W$ consist of all runs in which action 0 is played in stage $-2$ or in stage $-1$. 
Let $\widehat Q_{-2}=\{0\}\times A\times A$,
$\widehat Q_{-1}=\{0\}\times A$, and $\widehat Q_\ell=\emptyset$ for all $\ell\notin\{-2,-1\}$. 
So, \Eqref{eq-Wunion} is satisfied, 
$\widehat Q_{\geq -2}=\widehat Q_{-2}\cup \widehat Q_{-1}$, 
and $\widehat Q_{\geq -1}=\widehat Q_{-1}$. 

Consider a position $p_{-1}=(\ldots,a_{-4},a_{-3},0)$ in stage $-1$ of the game $G$, which ends with action 0 in stage $-2$. As discussed above, $p_{-1}$ can be seen as the position $(0)$, consisting of action 0, in stage $-1$ of the game $\widehat G_{\geq -2}$. So regardless of the actions $a_{-1}$ and $a_0$, the sequence $(0,a_{-1},a_0)$ will belong to $\widehat Q_{-2}\subseteq \widehat Q_{\geq -2}$. That is, $p_{-1}$ is winning for player 1 in the game $\widehat G_{\geq -2}$.

Also, $p_{-1}$ can be seen as the initial position in stage $-1$ of the game $G_{\geq -1}$. Note that $p_{-1}$ is winning for player 2 in the game $\widehat G_{\geq -1}$, because by playing action 1 in stage $-1$, player 2 can avoid $\widehat Q_{-1}=\widehat Q_{\geq -1}$. 

In conclusion, we have $w(p_{-1})=-1$. $\blacklozenge$\medskip

\noindent\textbf{Step 2: Properties of $w(p_n)$.} By the definitions, in any position $p_n=(\ldots,a_{n-2},a_{n-1})$ in any stage $n\in\dZ_0$:
\begin{itemize}
	\item[(i)] If $w(p_n)=m\in\dZ_0$, and thus $m\leq n$, then the position $p_n$ is winning for player 2 in the game $\widehat G_{\geq k}$ for all $k=m,\ldots,n$ and winning for player 1 in the game $\widehat G_{\geq k}$ for all $k<m$. 
	\item[(ii)] If $w(p_n)=-\infty$ then the position $p_n$ is winning for player 2 in the game $\widehat G_{\geq k}$ for all $k\leq n$. 
	\item[(iii)] Assume that $n$ is odd, so player 2 is the active player. Then, for all actions $a_n$ we have $w(p_{n},a_n)\geq w(p_n)$, so player 2's actions can keep equality at best. But there is an action $a_n\in A$ such that $w(p_{n},a_n)=w(p_n)$. Indeed, if $w(p_n)=m\in\dZ_0$ then $a_n$ can be any optimal action%
\footnote{
An action $a_n$ is \emph{optimal} for player~$i$ at position $p_n$ in the game $\widehat G_{\geq m}$
if (i) $p_n$ is a winning position for player~$i$ in $\widehat G_{\geq m}$,
(ii) $p_n$ is controlled by player~$i$, and
(iii) $a_n$ is part of a winning strategy of player~$i$ when the starting position is $p_n$.
}
for player 2 in the game $\widehat G_{\geq m}$. And if $w(p_n)=-\infty$, then by the finiteness of the action set $A$, there is an action $a_n$ such that $a_n$ is optimal in the game $\widehat G_{\geq m}$ for infinitely many $k\leq n$, ensuring that $w(p_{n},a_n)=w(p_n)=-\infty$.
 \item[(iv)] Assume that $n$ is even, so player 1 is the active player. 
 If $n<0$, then for every action $a_n\in A$ we have $w(p_{n},a_n)\leq w(p_n)$. 
 And if $n=0$, then for every action $a_0\in A$ the induced run 
 $(p_0,a_0)$ 
 does not end with a reversed position in $\widehat Q_{m}$ for any $m \in \color{black}\dZ_0$ with $m\geq w(p_0)$.
 \item[(v)] If $w(p_n)=-\infty$, then the position $p_n$ is winning for player 2 in the original game $G$.\smallskip\\ 
 Indeed, when starting in the position $p_n$, by (iii) and (iv), player 2 can make sure that the value of $w$ remains $-\infty$ in all stages $n,\ldots,-1,0$. In particular, for the position $p_0=(p_n,a_n,\ldots,a_{-1})$ in stage 0 we have $w(p_0)=-\infty$. And then by (iv) again, the induced run $r=(p_n,a_n,\ldots,a_{-1},a_0)$ does not end with a reversed position in $\widehat Q_{m}$ for any $m \in \color{black}\dZ_0$. That is,  $r$ does not end with a reversed position in $\widehat Q$. Hence, by \Eqref{eq-Wunion} we obtain $r\notin W$, as desired.
\end{itemize} 

\noindent\textbf{Step 3: Finding the equilibrium.} For each odd stage $n\in\dZ_0$, let $R_n$ denote the set of runs $r\in R$ such that, in each odd stage $m\in\{n,\ldots,-3,-1\}$, player 2's action satisfies (iii): $w(r_{<m+1})=w(r_{<m})$. The sets $R_n$ have the following properties:
\begin{itemize}
    \item The sequence $R_n$ is nonincreasing: $R_0\supseteq R_{-1}\supseteq R_{-2}\supseteq \cdots$\smallskip\\
    Indeed, this follows from the definition of $R_n$.
    \item  Each $R_n$ is nonempty.\smallskip\\
    Indeed, consider an arbitrary position $p_n$ is stage $n$. We can extend $p_n$ by always adding an action as in (iii) in odd stages and an arbitrary action in even stages. The run obtained this way belongs to $R_n$.
    \item  Each $R_n$ is compact.
    \smallskip\\
    Since $R_n\subseteq R$ and $R$ is compact and metrizable, we only need to show that $R_n$ is sequentially closed. Let $(r_k)_{k=1}^\infty$ be a sequence of runs in $R_n$ converging to $r$. We need to prove that $r\in R_{n}$.
    \smallskip\\
    \emph{Step 1:} Consider a stage $m\in\{n,\ldots,-1,0\}$. 
    We show that $\lim_{k\to\infty}w((r_k)_{<m}) = w(r_{<m})$. 
    
   Let $\ell\leq m$ and consider the game $\widehat G_{\geq\ell}$. Recall that the game $\widehat G_{\geq \ell}$ is  played in stages $\ell,\ldots,-1,0$ and its winning condition $\widehat Q_{\geq\ell}$ is independent of the actions chosen before stage $\ell$. Since $r_k\to r$ as $k\to\infty$, for large $k\in\N$ the actions in stages $\ell,\ldots,m$ along $r_k$ coincide with those along $r$. Hence, for large $k\in\N$, the position $(r_k)_{<m}$ is winning for player 2 in the game $\widehat G_{\geq\ell}$ if and only if the position $r_{<m}$ is winning for player 2 in the game $\widehat G_{\geq\ell}$. By \Eqref{eq-w}, it follows that $\lim_{k\to\infty}w((r_k)_{<m}) = w(r_{<m})$. 
    \smallskip\\
   \emph{Step 2:} We show that $r\in R_n$. 

   Consider an odd stage $m\in\{n,\ldots,-3,-1\}$. We need to show that $w(r_{<m+1})=w(r_{<m})$. 

   As $r_k\in R_n$ for each $k\in\N$, we have $w((r_k)_{<m+1})=w((r_k)_{<m})$. Hence, by Step 1, we obtain $w(r_{<m+1})=w(r_{<m})$, as desired.
    \end{itemize}
It follows from Cantor's intersection theorem that the intersection $R_0\cap R_{-1}\cap R_{-2}\cap \cdots $ is nonempty. Let the run $r=(\ldots,a_{-1},a_0)$ be an element of this intersection. Along $r$, by (iii) and (iv), the value of $w$ is nonincreasing, and the value of $w$ is at most $n$ in any stage $n\in\dZ_0$. Hence, along $r$ the value of $w$ is always $-\infty$. By (v), in the game $G$, player 2 is winning in each position along $r$. Hence, by Part 2 of Lemma \ref{lemma-howwins}, there is a strategy profile~$s$ such that $(s,r)$ is an equilibrium, and $r\notin W$. Therefore, the proof is complete.
\end{proof}

\begin{remark}[\textbf{Alternative definition of $(R_n)_{n \in \dZ_0}$ in Step~3 of the proof of Lemma~\ref{theorem-open}}]\rm 
In Step~3 of the proof of Lemma~\ref{theorem-open}, 
to identify an equilibrium run,
we defined a nested sequence $R_0\supseteq R_{-1}\supseteq R_{-2}\cdots$ of sets of runs. One advantage of the definitions of these sets is that $R_{n-1}$ eliminates runs from $R_n$ purely based on the action in stage $n-1$. There are alternative possible definitions to these sets.
For example, each $R_n$ could be the set of runs $r$ such that each position $r_{<n},\ldots,r_{<-1},r_{<0}$ is winning for player 2 in the game $G_{\geq n}$. This would lead to a somewhat easier definition (as the function $w$ is not involved), but with this definition $R_{n-1}$ might eliminate runs from $R_n$ based on the actions in stages $n\ldots,-1,0$, which is arguably less constructive. $\blacklozenge$
\end{remark}

\begin{remark}[\textbf{On the implication of the second condition in Lemma \ref{theorem-open}}]\label{rem-bothposs}\rm 
In Example~\ref{ex-winloseruns},
$A=\{0,1\}$ and $W$ is the set of all runs in which player 2 played action~1 at least once. 
The set $W$ is open, and this game belongs to Part~2 in Lemma \ref{theorem-open}.\footnote{Indeed, for instance, let $\widehat Q$ be the set of all reversed positions that start in an odd stage with action 1 (recall that player 2 is active in odd stages). Then, for each odd $n\in\dZ_0$ we have $\widehat Q_n=
\{1\}\times A^{-n}$,
and for each even $n\in\dZ_0$ we have $\widehat Q_n=\emptyset$. 
Hence, player 2 wins trivially in the game $\widehat G_{\geq 0}$ (as $\widehat Q_{\geq 0}$ is empty) and player 2 can win in the game $\widehat G_{\geq n}$ for each $n<0$ by playing action 1 at any odd stage, say at stage $-1$.}
Moreover,
the equilibrium $(s^0,r^0)$
satisfies $r^0 \not\in W$.
However, as we have seen,
this game admits an additional equilibrium, 
$(s^1,r^1)$,
such that $r^1 \in W$.
$\blacklozenge$
\end{remark}

\subsection{Existence of equilibrium for winning sets of Borel-rank at most 2}\label{sec-rank2}

Recall that, in the product topology of the set $R$ of runs, a Borel set has rank at most 2 (so rank 1 or rank 2) exactly when it can be written as a countable union of closed sets or as a countable intersection of open sets. 
The next theorem answers the question of existence of equilibrium in the affirmative for two-player win-lose games,
provided that the winning set has Borel-rank at most 2. In particular, this is the case if the winning set is open or closed.
In Section~\ref{sec-rank3}, we show that an equilibrium need not exist when the winning set has Borel-rank $3$.

We will use auxiliary games where the winning set of player 1 is not $W$ but some other set $U\subseteq R$. To avoid confusion, we will denote this game by $G(U)$.

\begin{theorem}[\textbf{Existence of equilibrium for winning sets of Borel-rank at most 2}]\label{theorem-rank2} Consider a two-player win-lose game $G$. Assume that player 1's winning set $W$ has Borel-rank at most 2 in the product topology. Then the game $G$ admits an equilibrium, 
and hence, also a strong equilibrium.

More precisely, when $W$ is a countable intersection of open sets\footnote{A similar statement can be formulated when $W$ is countable union of closed sets, because in that case the winning set $R\setminus W$ of player~2 is a countable intersection of open sets.}, we have two cases: 
\begin{enumerate}
    \item[(1)] Assume that there exists an open set $O\supseteq W$ of runs such that the game $G(O)$ admits an equilibrium $(s,r)$ with a winning run for player 2, i.e., $r\notin O$. Then, $(s,r)$ is an equilibrium in the game $G$ and $r\notin W$.
    \item[(2)] Assume that for every open set $O\supseteq W$ of runs, 
    all equilibria $(s,r)$ of
    the game $G(O)$ satisfy 
    $r\in O$. Then, the game $G$ admits an equilibrium $(s',r')$ with a winning run for player 1, i.e., $r'\in W$.
\end{enumerate}
\end{theorem}

\begin{proof} 
We prove the statement when $W$ is a countable intersection of open sets.\medskip

\noindent\textbf{Proof of Part 1:} Suppose that an open set $O$ and an equilibrium $(s,r)$ in the game $G(O)$ exist as in Part 1. This means that the game $G(O)$ belongs to Part 2 of Lemma~\ref{theorem-open}.

Since $r\notin O$ and $O\supseteq W$, we have $r\notin W$. Thus, it remains to verify that $(s,r)$ is an equilibrium in the game $G$ (with winning set $W$ for player 1 and winning set $R\setminus W$ for player 2). As $r\notin W$, player~2 has no incentive to deviate in any position along $r$. Now suppose that player 1 deviates from $r$ in one of the positions along $r$, which leads to the run $r'$. Because $(s,r)$ is an equilibrium in $G(O)$ with $r\notin O$, we must also have $r'\notin O$, and hence $r'\notin W$. This means that player~1's deviation is not profitable in the game $G$. 
Therefore, $(s,r)$ is an equilibrium in the game $G$, as claimed.\medskip

\noindent\textbf{Proof of Part 2:} As $W$ is a countable intersection of open sets, we can write $W=\bigcap_{k=1}^\infty O_k$ where each $O_k\subseteq R$ is open. We can assume w.l.o.g.~that the sequence $(O_k)_{k=1}^\infty$ is non-increasing: $O_1\supseteq O_2\supseteq O_3\supseteq \cdots $; indeed, otherwise replace each $O_k$ with $O_1\cap\cdots\cap O_k$.

For each $k\in\dN$, there is a stage $n_k\in\dZ_0$ such that in the game $G(O_k)$ each position in stage $n_k$ is winning for player 1. Indeed, this follows as $O_k$ is an open set, and by the assumption of Part 2, the game $G(O_k)$ belongs to Part 1 of Lemma \ref{theorem-open} (regardless of the set $\widehat Q$ and the resulting auxiliary games $\widehat G_{\geq n}$). 

For each stage $n\in\dZ_0$ and position $p_n=(\ldots,a_{n-2},a_{n-1})$, let\footnote{Note that, in this proof, $w(p_n)$ is a natural number or $+\infty$, related to the indices of the sequence $O_1,O_2,\ldots$ of open sets, whereas in the proof of Theorem \ref{theorem-open}, $w(p_n)$ was a stage (a nonnegative integer) or $-\infty$ related to winning stages for a given open set.}
\[w(p_n)\,=\,\sup\Big\{k\in\dN\,\colon\,\text{the position $p_n$ is winning for player 1 in the game }G(O_k)\Big\},\]
with the convention that $w(p_n)=0$ if $p_n$ is not winning for player 1 in the game $G(O_k)$ for any $k\in\dN$.\footnote{This is equivalent to the condition that $p_n$ is not winning for player 1 in $G(O_1)$, as $O_1\supseteq O_k$ for all $k\in\N$.} Hence:
\begin{itemize}
	\item[(i)] If $w(p_n)=m\in\dN$, then $p_n$ is winning for player 1 in the game $G(O_k)$ for all $k=1,\ldots,m$ and winning for player 2 in the game $G(O_k)$ for all $k>m$. 
    This statement holds because 
    the sequence $(O_k)_{k=1}^\infty$ is non-increasing. 
	\item[(ii)] If $w(p_n)=\infty$, then $p_n$ is winning for player 1 in the game $G(O_k)$ for all $k\in\dN$.
	\item[(iii)] If $n\leq n_k$, then $w(p_n)\geq k$. Indeed, the position $p_n$ is winning for player 1 in the game $G(O_k)$: player 1 can play arbitrarily from stage $n$ until stage $n_k$, and then by the definition of $n_k$, she can guarantee a win.
	\item[(iv-1)] Assume that $n<0$ is even, so player 1 is the active player. Then there is an action $a_n\in A$ such that $w(p_{n},a_n)=w(p_n)$.\smallskip\\
    Indeed, if $w(p_n)=m\in\dN$, then $a_n$ can be any action for player 1 that is optimal in the game $G(O_m)$. And if $w(p_n)=\infty$, then by the finiteness of the action set $A$, there is an action $a_n$ such that $a_n$ is optimal in the game $G(O_k)$ for infinitely many $k\in\dN$, ensuring that $w(p_n,a_n)=w(p_n)=\infty$.
   \item[(iv-2)] Now assume that $n=0$, so player 1 is the active player. Then, similarly, there is an action $a_0\in A$ such that the induced run $r=(p_0,a_0)$ belongs to $O_k$ for all $k\in\N$ with $k\leq w(p_1)$.
 \item[(v)] Assume that $n$ is odd, so player 2 is the active player. Then for every action $a_n\in A$ we have $w(p_{n},a_n)\geq w(p_n)$.
 \item[(vi)] If $w(p_n)=\infty$, then $p_n$ is winning for player 1 in the original game $G$.\smallskip\\
  Indeed, when starting in the position $p_n$, by (iv-1) and (v), player 1 can make sure that the value of $w$ remains $\infty$ in all stages $n,\ldots,-1,0$. In particular, for the position $p_0=(p_n,a_n,\ldots,a_{-1})$ in stage 0 we have $w(p_0)=\infty$. And then by (iv-2), player 1 has an action $a_0$ such that the induced run $r=(p_n,a_n,\ldots,a_{-1},a_0)$ belongs to $O_k$ for all $k\in\N$. That is, $r\in W$, as desired.
\end{itemize} 

To complete the proof, we adapt the arguments from the proof of Lemma \ref{theorem-open}.
For each even stage $n\in\dZ_0$, let $R_n$ denote the set of runs $r\in R$ such that
$w(r_{<m+1})=w(r_{<m})$ for each even stage $m\in\{n,\ldots,-2,0\}$. The set $R_n$ is nonempty by (iv-1) and compact. Moreover, the sequence $R_{0},R_{-2},R_{-4},\ldots$ is nonincreasing. Hence, by Cantor's intersection theorem, there exists a run $r \in \bigcap_{n \in \dZ_0, \text{even}} R_{n}$.
By construction and (v), along $r$ the value of $w$ is nondecreasing, and by (iii), $w$ takes on arbitrarily high values. Hence, along $r$ the value of $w$ is always $\infty$. Consequently, by (vi), each position along $r$ is winning for player 1 in the game $G$. Note that by (iv-2), we can assume w.l.o.g.~that player 1's action $a_0$ in stage 0 is such that $r\in O_k$ for all $k\in\N$, that is, $r\in W$. By Part 2 of Lemma \ref{lemma-howwins}, the proof is complete.
\end{proof}

\subsection{Non-existence of equilibrium for winning sets of Borel-rank 3}\label{sec-rank3}

In this section, we show that the existence result established in Theorem \ref{theorem-rank2} cannot be extended to winning sets of Borel-rank 3, not even to the class $\Delta_3$ of Borel sets.

The following example exhibits a two-player win-lose game $G$ 
where the players have only two actions, player 1's winning set $W$ belongs to $\Delta_3$, and yet, $G$ admits no equilibrium, and hence, no strong equilibrium either.

\begin{example}[\textbf{A game with a winning set $W \in \Delta_3$ and no equilibrium}]
\label{ex-counter}\rm 
First we define the game, then we show that it admits no equilibrium, and finally we prove that player 1's winning set $W$ belongs to $\Delta_3$.\medskip

\noindent\textsc{\underline{The game.}} Consider the following two-player win-lose game $G$. The set of players is $I=\{1,2\}$, and the action set is $A=\{0,1\}$. To complete the definition of the game $G$, it remains to define the winning set $W$ for player~1. 

A run $r=(\ldots,a_{-2},a_{-1},a_0)\in R$ is called an \emph{almost single-action run} for player~1 if player~1 played one of the actions only finitely many times 
(possibly $0$)
along $(\ldots,a_{-4},a_{-2},a_0)$. Let $R^{f1}$ denote the set of all runs $r\in R$ that are almost single-action for player~1. 
If $r\in R^{f1}$, the action that player~1 played only finitely many times along $r$ is called the \emph{finite-action} of player~1 along $r$. 
Since $|A|=2$,
the finite-action of player~1 is unique.

For each run $r=(\ldots,a_{-2},a_{-1},a_0)\in R$ and each action $a \in A=\{0,1\}$, we denote by $m_1(r,a)$ the first even stage from which onward player~1 only played the action $a$:
\[ m_1(r,a) \,=\, \min\Big\{ k \in \{\ldots,-4,-2,0\} \,\colon\, a_{k}=a_{k+2}=\cdots=a_0= a\Big\},\]
with the convention that $m_1(r,a) = \infty$ if $a_0 \neq a$  
and $m_1(r,a)\, =\, -\infty$ if player~1 always selected the action $a$ along $r$.

For player~2, almost single-action runs, the set $R^{f2}$, the finite-action, and the quantity $m_2(r,a)$ are defined 
analogously.

We turn to define the winning set $W$ for player~1. Let $W \subset R$ be the set that consists of the following runs:
\begin{itemize}
\setlength\itemsep{0cm}
\item[(i)]
All runs $r$ such that $r\in R^{f1}$, $r\in R^{f2}$, and 
$m_1(r,a) < m_2(r,a')$, where $a$ is the finite-action of player~1 and $a'$ is the finite-action of player~2.
\item[(ii)] 
All runs $r$ such that $r\not\in R^{f1}$ and $r\notin R^{f2}$ in which $m_1(r,1) < m_2(r,1)$. 
\item[(iii)]
All runs $r$ such that $r\in R^{f1}$, $r\not\in R^{f2}$, and $m_1(r,a) < m_2(r,a)$, where $a$ is the finite-action of player~1.
\item[(iv)]
All runs $r$ such that $r\not\in R^{f1}$, $r\in R^{f2}$, and $m_1(r,a) < m_2(r,a)$, where $a$ is the finite-action of player~2.
\end{itemize}\vspace{0.1cm}

Note that each run can satisfy at most one of the items (i)--(iv).
\bigskip

\noindent\textsc{\underline{The game admits no equilibrium.}} We show that, for every run $r\in R$, there exists a position $p_{k_1}$ along $r$ that is winning for player~1 and there exists a position $p_{k_2}$ along $r$ that is winning for player~2. This implies by
Part~1 in 
Lemma \ref{lemma-howwins} that the game $G$ has no equilibrium. 

We will prove this claim for player~1; the proof for player~2 is analogous. Let $r=(\ldots,a_{-1},a_0)\in R$ be a run. We distinguish a number of cases. 
\begin{itemize}
\setlength\itemsep{0cm}
\item Assume that $r\in R^{f1}$ and $r\in R^{f2}$. Let $a$ denote the finite-action of player~1 and $a'$ denote the finite-action of player~2. Then $-\infty<m_1(r,a)$ and $-\infty<m_2(r,a')$. Hence, along $r$, player~1 has a winning position $p_n=(\ldots,a_{n-2},a_{n-1})$ in the even stage $n=m_2(r,a')-1$: indeed, in the odd stage $m_2(r,a')-2$ player~2 did not play the action $a'$, so if player~1 plays her finite-action $a$ in stages $n,n+2,\ldots,0$ then she wins due to (i). 

\item Assume that $r\not\in R^{f1}$ and $r\not\in R^{f2}$. Then $-\infty<m_1(r,1)$ and $-\infty<m_2(r,1)$. Hence, along $r$, player~1 has a winning position $p_n=(\ldots,a_{n-2},a_{n-1})$ in the even stage $n=m_2(r,1)-1$: indeed, in the odd stage $m_2(r,1)-2$ player~2 did not play the action 1, so if player~1 plays the action 1 in stages $n,n+2,\ldots,0$ then she wins due to (ii).

\item Assume that $r\in R^{f1}$ and $r\not\in R^{f2}$. Let $a$ denote the finite-action of player~1. Then $-\infty<m_1(r,a)$ and $-\infty<m_2(r,a)$. Hence, along $r$, player~1 has a winning position $p_n=(\ldots,a_{n-2},a_{n-1})$ in the even stage $n=m_2(r,a)-1$: indeed, in the odd stage $m_2(r,1)-2$ player~2 did not play the action $a$, so if player~1 plays her finite-action $a$ in stages $n,n+2,\ldots,0$ then she wins due to (iii). 

\item Assume that $r\not\in R^{f1}$ and $r\in R^{f2}$. The proof is similar to that of the previous case.
\end{itemize}
Hence, the game admits no equilibrium.\medskip

\noindent\textsc{\underline{The winning set $W$ belongs to $\Delta_3$.}} We define a few sets of runs.
\begin{itemize}
	\item Let $R^k[i,a]=\{r\in R\colon m_i(r,a)=k\}$, for each action $a\in A$ and either for (i) player $i=1$ and even stage $k\in\dZ_0$ or $k=\infty$, or for (ii) player $i=2$ and odd stage $k\in\dZ_0$ or $k=\infty$. Each set $R^k[i,a]$ is open (and also closed). Indeed, for example, $R^{-2}[1,0]$ is determined by the actions $a_{-4}=1$, $a_{-2}=0$, $a_0=0$. The set \[\Big\{r\in R\colon m_1(r,a)=k_1,\ m_2(r,a')=k_2,\ k_1<k_2,\ k_1\in\dZ_0,\ k_2\in\dZ_0\cup\{\infty\}\Big\}\]
is a (countable) union of sets $R^k[i,a]$, and hence open.
\item 
For each player $i=1,2$, each action $a\in A$, and each $k\in\dN$,
let $R^{\leq k}[i,a]$ denote the set of runs $r\in R$,  where along $r$ player $i$ plays actions action $a$ at most $k$ times. Each set $R^{\leq k}[i,a]$ is closed. 
Hence, for the set $R^{fi}$ of almost-single action runs for player $i$, we have $R^{fi}=\bigcup_{a\in A}\bigcup_{k\in\dN}R^{\leq k}[i,a]$. 
Thus, $R^{fi}$ is a countable union of closed sets: a $\Sigma_2$-set. 
Therefore, 
$R\setminus R^{fi}$ is a countable intersection of open sets: a $\Pi_2$-set.
\end{itemize}
It follows that the set of runs under each case in (i)-(iv) in Example \ref{ex-counter} is a finite intersection of $\Sigma_2$-sets and $\Pi_2$-sets, and hence a $\Delta_3$-set. 
Since $W$ is the union of the sets of the cases (i)-(iv), 
$W$ is a $\Delta_3$-set. $\blacklozenge$
\end{example}

\subsection{Equilibrium and value in segments}\label{sec-segwinlose}

In this section, we discuss the implications of our results for a given segment in a two-player win-lose game. As we will see, in a given segment, we can define the value and winning strategies in the usual way. Recall 
Definition \ref{def-gamerestr}
of the game $G_\Omega$ and the set of strategies $S_i(\Omega)$.

\begin{lemma}[\textbf{Same winning player in a segment}]
\label{lem-same winning}
Consider a two-player win-lose game $G$ and a segment $\Omega$. Assume that $(s,r)$ and $(s',r')$ are two equilibria such that $r,r'\in\Omega$; in other words, they are equilibria in the game $G_\Omega$. Then, $r$ and $r'$ are winning for the same player.
\end{lemma}

\begin{proof}
If $r$ and $r'$ belong to the same segment, then they have a common position $r_{<n}=r'_{<n}$ in some stage $n\in\dZ_0$, which is either winning for player 1 or for player 2. The statement follows.
\end{proof}

\begin{definition}[\textbf{Winning strategies in a segment and determined segment}]\rm 
Consider a two-player win-lose game $G$ and a segment $\Omega$. A strategy $s_1\in S_1(\Omega)$ of player 1 is \emph{winning in $G_\Omega$} if for any strategy $s_2\in S_2(\Omega)$, the (unique) run in $\Omega$ that is consistent with $(s_1,s_2)$ belongs to $W$. Winning strategies for player 2 are defined similarly. The game \emph{$G_\Omega$ is determined} if one of the players has a winning strategy in $G_\Omega$.
\end{definition}

By Lemma~\ref{lem-same winning},
if $G_\Omega$ is determined, then the value of $G_\Omega$ exists, and it is equal to 1 or 0. 
The next result follows directly.

\begin{lemma}[\textbf{Equilibrium implies determinacy}]
Consider a two-player win-lose game $G$ and a segment $\Omega$. Assume that $(s,r)$ is an equilibrium in $G_\Omega$. If $r\in W$ then $s_1$ is winning in $G_\Omega$ for player 1, and if $r\in \Omega\setminus W$ then $s_2$ is winning in $G_\Omega$ for player 2. Hence, $G_\Omega$ is determined.
\end{lemma}

The next statement follows from Theorem \ref{theorem-rank2}.

\begin{theorem}[\textbf{Existence of a determined segment}]\label{thm-detsegrank2} Consider a two-player win-lose game~$G$. If $W$ has Borel-rank at most 2, then there is a segment $\Omega$ for which the game $G_\Omega$ is determined.\footnote{Recall 
that different determined segments can be winning for different players,
see Remark \ref{rem-bothposs}.}
\end{theorem}

The following example shows that the conclusion of Theorem \ref{thm-detsegrank2} is generally not true for all segments, not even if $W$ is open.

\begin{example}[\textbf{Non-determined segment under an open winning set}]\rm
Consider the game in Example \ref{ex-counter} and the specific segment $\Omega$ in which: (1) all runs are almost single-action for both players, and (2) action 1 is the finite action of both players. 

We change the winning set $W$ of Example \ref{ex-counter} (which has rank 3) into an open set $W'$ in such a way that $W\cap\Omega = W'\cap \Omega$, and hence $G_\Omega$ is non-determined. Let the open set $W'$ be given by the reversed positions $q$ such that $q$ starts at an odd stage, say at stage~$n$, and in $q$ player 2 plays action 0 in stage $n$ and player 1 plays action 1 in the even stages $n+1,\ldots,-2,0$. Within $\Omega$, the set $W'$ expresses that player 1 wins because she is the first to  use her finite action exclusively after some stage. Thus, $W\cap\Omega = W'\cap \Omega$, and hence $G_\Omega$ is not determined 
when the 
winning set
is
$W'$. $\blacklozenge$
\end{example}

\section{Multi-player non-zero-sum games}\label{sec-multiplayer}

In this section we turn to multi-player games with infinite past. In light of Example~\ref{ex-counter}, the existence of an equilibrium can only be guaranteed under relatively strong conditions. 
We establish the existence of equilibrium when the payoff functions of the players are continuous.

A payoff function $u_i:R\to\R$ of player $i\in I$ is \emph{continuous} in the product topology on the set $R$ of runs, if for every run $r=(\ldots,a_{-1},a_0)$ and every $\ep>0$ there is a stage $n\in\dZ_0$ with the following property: 
$|u_i(r)-u_i(r')|\leq \ep$
for every
run $r'$ 
such that $r_{< n} = r'_{< n}$.

The analysis of games with continuous payoffs 
turns out to be
more complicated than one would probably expect, 
for two reasons. We discuss each reason separately.\medskip

\noindent\textbf{I: Continuity of the function that assigns the active player to each position.} 
If two runs are close to each other, then continuity of the payoff function ensures that the payoffs are also close to each other. Yet, if the active players are different along these runs, then the strategic considerations along these two runs can be very different from each other, and the payoff difference between the two runs is rather irrelevant. Therefore, to have a meaningful consequence of payoff continuity, we also need to require that, for each stage $n\in\dZ_0$, the function $\iota_n:P_n\to I$, which assigns the active player to each position in this stage, is continuous. We will show in Example \ref{ex-cont-noeq} that continuity of $\iota_n$ is indeed essential.

In fact, for a given stage $n$, the function $\iota_n$ is continuous if and only if there is a stage $m\leq n$ (where $m$ may depend on $n$) such that the value of the function $\iota_n$ does not depend on the actions before stage $m$: if $p_n$ and $p'_n$ are two positions such that in each stage $k\in\{m,m+1,\ldots,n-1\}$ the actions in $p_n$ and $p'_n$ coincide, then $\iota(p_n)=\iota(p'_n)$.\footnote{Indeed, the existence of such a stage $m$ implies continuity of $\iota_n$. Now we argue that if $\iota_n$ is continuous, then there is such a stage $m$. Since $\iota_n$ is continuous, for each player $i\in I$, the set $\iota^{-1}(i)$ is an open subset of $P_n$. Hence, each $\iota^{-1}(i)$ is also closed, and thus clopen. It is known that a clopen subset of an infinite product of finite sets can only depend on finitely many coordinates (e.g., Dubins and Savage [2014], or Lemma 3 in Flesch et al. [2019]).} For instance, $\iota_n$ is continuous if the identity of the active player only depends on the actions played in stages $n-1$ and $n-2$. 
\medskip

\noindent\textbf{II. Existence of a consistent run.} We discussed in Section \ref{sec-backind} that backward induction can deliver a strategy profile that admits no consistent run. This issue remains present even if the game has continuous payoffs, cf.~the discussion of Example \ref{ex-norun} in Section \ref{sec-backind}. Thus, as before, to guarantee the existence of a consistent run, our proof uses Cantor’s intersection theorem. 

\begin{theorem}[\textbf{Existence of equilibrium for continuous payoffs}]\label{theorem-cont}
Consider a game $G$ with infinite past. Let the set $P_n$ of positions in stage $n$, for any stage $n\in\dZ_0$, and the set $R$ of runs be endowed with the product topology. Assume that:
\begin{itemize}
    \item[(C1)] The function $\iota_n:P_n\to I$, which assigns the active player to each position in stage $n$, is continuous for each $n\in\dZ_0$.
    \item[(C2)] The payoff function $u_i:R\to\mathbb{R}$ of each player $i\in I$ is continuous.
\end{itemize} 
Then, the game $G$ admits an equilibrium.
\end{theorem}

\begin{proof}
For each stage $n\in\dZ_0$, let $R_n$ denote the set of runs $r=(\ldots,a_{-1},a_0)$ such that 
in the subgame $G[r_{<n}]$
there is a strategy profile $s^{r,n}$ in stages $n,\ldots,-1,0$ with the following properties: (1) $s^{r,n}$ is an equilibrium in this subgame, and 
(2)  $s^{r,n}$ induces the sequence $(a_n,\ldots,a_{-1},a_0)$. 
Thus,
the domain of $s^{r,n}$ is the set $A^{\leq -n}$ of sequences in $A$ of length at most $-n$, which is a finite set. The position $r_{<n}$ is needed in (1) to obtain the payoffs in the subgame $G[r_{<n}]$ through the payoff functions $(u_i)_{i\in I}$, and thus to verify the equilibrium conditions.

The sets $R_n$ have the following properties:
\begin{itemize}
    \item[(i)] The sequence $R_n$ is nonincreasing: $R_0\supseteq R_{-1}\supseteq R_{-2}\supseteq \cdots$\smallskip\\
    We show that $R_{n+1}\supseteq R_n$, where $n<0$. So, assume that $r=(\ldots,a_{-1},a_0)\in R_n$ with the corresponding strategy profile $s^{r,n}$. For the sequence $(a_{n+1},\ldots,a_{-1},a_0)$, the continuation of the strategy profile $s^{r,n}$ from stage $n+1$ satisfies (1) and (2) of the definition of $R_{n+1}$. Hence, $r\in R_{n+1}$, as desired.
    
    \item[(ii)]  Each $R_n$ is nonempty.\smallskip\\
    Indeed, $r=(\ldots,a_{-1},a_0)$ belongs to $R_n$ whenever $r_{<n}$ is an arbitrary position in stage $n$ and $(a_n,\ldots,a_{-1},a_0)$ is the sequence of actions induced in the subgame $G[r_{<n}]$ by an arbitrary equilibrium. Note that each subgame is a finite game, so it admits an equilibrium.
    
    \item[(iii)]  Each $R_n$ is compact.\smallskip\\
    Indeed, since $R_n\subseteq R$ and $R$ is compact and metrizable, we only need to show that $R_n$ is sequentially closed. Let $(r_k)_{k=1}^\infty$ be a sequence of runs in $R_n$ converging to $r$.
    We need to prove that $r\in R$.\smallskip\\
    Since $(r_k)_{k=1}^\infty$ converges to $r$, by taking a subsequence if necessary, we can assume that all $r_k$ and $r$ end with the same sequence $(a_n,\ldots,a_{-1},a_0)$ in stages $n,\ldots,-1,0$.\smallskip\\
    For each $k\in\N$, we have $r_k\in R_n$, and hence there is a strategy profile $s_k:=s^{r_k,n}$ in stages $n,\ldots,-1,0$ with the following properties: (1) $s_k$ is an equilibrium in the subgame $G[(r_k)_{<n}]$, where $(r_k)_{<n}$ is the position along $r_k$ in stage $n$, and (2) $s_k$ induces the sequence $(a_n,\ldots,a_{-1},a_0)$ in stages $n,\ldots,-1,0$.\smallskip\\
    Since the domain of each $s_k$ is the finite set $A^{\leq -n}$, we can assume, by taking a subsequence if necessary, that $s_k$ is a constant: $s_k=s$ for all $k\in\N$. \smallskip\\
    Thus: 
    \begin{itemize}
        \item Since $s_k=s$ is an equilibrium in the subgame $G[(r_k)_{<n}]$, for all $k\in\N$, using conditions (C1) and (C2) and that $(r_k)_{<n}$ converges to $r_{<n}$, we obtain that $s$ is an equilibrium in the subgame $G[r_{<n}]$.
        \item $s_k=s$ induces the sequence $(a_{n},\ldots,a_{-1},a_{0})$ in stages $n,\ldots,-1,0$, for all $k\in\N$.
    \end{itemize} 
\end{itemize}
By properties (i), (ii), and (iii), it follows from Cantor's intersection theorem that the intersection $R_0\cap R_{-1}\cap R_{-2}\cap \cdots $ is nonempty. Let the run $r=(\ldots,a_{-1},a_0)$ be an element of this intersection. For each $n\in\dZ_0$, as $r\in R_n$, we can fix a strategy profile $s^{r,n}$ as in the definition of $R_n$. Let $s$ be a strategy profile that satisfies the following properties for the position $r_{<n}$ along $r$ in any stage $n\in\dZ_0$:
\begin{itemize}
    \item[(1)] The strategy $s_{i_n}$ of the active player $i_n$ recommends to follow $r$ and thus to play the action $a_n$.
    \item[(2)] If player $i_n$ deviates from $a_n$ in the position $r_{<n}$, then the opponents of player $i$ switch to their strategies in $s^{r,n}$ from stage $n+1$ onward, in the remaining subgame. Since $s^{r,n}$, by the definition of $R_n$, is an equilibrium in $G[r_{<n}]$, this makes sure that player $i_n$'s deviation is not profitable.
\end{itemize}
By (1), the run $r$ is consistent with $s$. Thus, $(s,r)$ is an equilibrium. 
\end{proof}

One application of Theorem \ref{theorem-cont} is when the payoffs are reversed-time discounted. 
Such an evaluation was considered, e.g., by Puente [2006], and by Ray et al.~[2024] in their model of dynamic consumption.
For the same reasons as in the case of Theorem \ref{theorem-cont}, 
the proof of this application
does not follow from the usual arguments for the discounted payoff in the standard setting where the stages is the set $\N$ of natural numbers.

\begin{corollary}[\textbf{Games with reversed-time discounted payoffs}]\label{cor-reverse}
Consider a game with infinite past that satisfies condition (C1) of Theorem \ref{theorem-cont}. For each player $i\in I$, let $g_i:I\times A\to\R$ be a function that assigns an instantaneous payoff $g_i(j,a)$ to player $i$ whenever player $j\in I$ plays action $a\in A$. Assume that, along each run $r=(\ldots,a_{-1},a_0)$ where the sequence of active players is denoted by $(\ldots,j_{-1},j_0)$, player $i$'s payoff is 
\begin{align}u_i(r)\,&=\,(1-\delta)\cdot \sum_{n=-\infty}^0 \delta^{-n}g_i(j_n,a_n)\label{eq-revdisc}\\
&=\,(1-\delta)\cdot(g_i(j_0,a_0)+\delta \cdot g_i(j_{-1},a_{-1})+\delta^2\cdot g_i(j_{-2},a_{-2})+\cdots),\nonumber
\end{align}
where $\delta\in(0,1)$ is a discount factor. Then the game admits an equilibrium.
\end{corollary}

\begin{remark}[\textbf{Continuous payoffs in a segment}]\rm 
Under conditions (C1) and (C2), Theorem \ref{theorem-cont} implies that there is a segment $\Omega$ such that the game $G_{\Omega}$ admits an equilibrium. In general, not all segments admit an equilibrium. Consider the 1-player game in which the action set is $A=\{0,1\}$, each action $a\in A$ gives an instantaneous payoff equal to $a$, and the payoff is reversed-time discounted as in \Eqref{eq-revdisc}. Then, in the unique equilibrium player 1 always plays action 1. $\blacklozenge$
\end{remark}

The following example shows that condition (C1) in Theorem \ref{theorem-cont} is essential: condition (C2) alone, even if the payoffs arise through  reversed-time discounting, is generally not sufficient to guarantee the existence of an equilibrium.

\begin{example}[\textbf{Reversed-time discounted payoffs but no equilibrium}]\label{ex-cont-noeq}\rm 

There are two players in the game and two actions: $I=\{1,2\}$ and $A=\{0,1\}$. Before stage 0, the active player is always player 1. In stage 0, player 2 is the active player in the position $p_0=(\ldots,1,1)$ in which player 1 has always played action 1 before stage 0, and player 1 is the active player in all other positions in stage 0.  
Thus, in this game, condition (C1) of Theorem \ref{theorem-cont} fails.

The payoffs are as reversed-time discounted as in \Eqref{eq-revdisc} with $g_1(1,a)=g_1(2,a)=a$ for player 1 and $g_2(2,a)=1-a$ for player 2 (with $g_2(1,a)$ being arbitrary). This means that, as far as the instantaneous payoffs are concerned, player 1 wants to choose action 1 and player 2 wants to choose action 0.\footnote{In fact, this game can be seen as a zero-sum game.} 
Note that condition (C2) of Theorem \ref{theorem-cont} holds in this game. 

One can check similarly to Example \ref{ex-noeq} that this game has no equilibrium. The crucial point is that if player 1 always chooses her best action, namely action 1, then player 2 becomes active in stage 0 and chooses action 0, which gives player 1 a bad instantaneous payoff, namely payoff 0, in stage 0. Hence, player 1 should choose action 0 once before stage 0, but there is no optimal stage to do so, as player 1 would like to do it as early as possible. $\blacklozenge$
\end{example}

\section{Discussion}
\label{sec-discussion}

\noindent\textbf{Infinite action sets.} We assumed in Definition \ref{def-game} that the action set $A$ is finite. We used this assumption multiple times when we applied Cantor's intersection theorem (cf.~Lemma \ref{theorem-open}, Theorems \ref{theorem-rank2} and \ref{theorem-cont}, and therefore also Corollary \ref{cor-reverse}). If $A$ is infinite, the statements and the corresponding proofs fail: $A$ is not compact in the discrete topology, and hence the set $R$ of runs is no longer compact in the product topology.
Therefore, 
we cannot apply Cantor's intersection theorem.

For instance, consider the following game $G$: The action set is $A=\dZ_0$ and player 1's winning set is $W=R$, so player 1 is winning along each run. Recall that player 2 is active in odd stages. 
For each odd stage $n\in\dZ_0$ denote by $\widehat Q_n$ the set of reversed positions $(a_n,\ldots,a_{-1},a_0)$ with $a_{-1}\geq n$.
We thus obtain a representation as in \Eqref{eq-Wunion} for $W=R$. 

As $W=R$, all equilibria have a run winning for player 1. But for every $n\in\dZ_0$, player~2 has a winning strategy in $\widehat G_{\geq n}$: she can play action $a_{-1}=n-1$ in stage $-1$. This demonstrates that Part 2 of Lemma \ref{theorem-open} is not valid when the action set $A$ is infinite.\medskip

\noindent\textbf{Games with infinite past and infinite future.}
In standard dynamic games, the set of stages is the set $\N$ of natural numbers. Therefore, one may wonder about the extension of the games in Definition \ref{def-game} to games where the future is also infinite, that is, when the set of stages is $\dZ$, the set of all integers. Since the model of these games is very similar to that in Definition \ref{def-game} and Section \ref{sec-model}, we do not provide a detailed description. The main difference is that now each subgame $G[p_n]$, where $p_n$ is a position in some stage $n\in\dZ$, has an infinite duration in stages $n,n+1,n+2,\ldots$.

The counter-example, Example \ref{ex-counter}, can be extended to this setup, by assuming that the actions in stages in $\N$ do not affect the payoffs.

It is an interesting question, and a possible line for future research, to determine to which degree our existence results can be extended to this model. 
One possible idea would be the following. Consider a position in stage 1. The resulting subgame is a standard game, played on the set of stages $\N$. Hence, one can try to find an equilibrium in the subgame by applying a result from standard dynamic games. This would allow to fix the payoffs after the action in stage 0: the payoff for a run $(\ldots,a_{-1},a_0)$ would be the equilibrium payoff found for the remaining subgame. This approach, however, encounters an additional difficulty: the payoffs fixed after the action in stage 0 must be a nice function of the run $(\ldots,a_{-1},a_0)$, e.g., 
a continuous function, otherwise our existence result is not applicable.\medskip

\noindent\textbf{Games with simultaneous moves.} In our model, in each position only one of the players is active. It would be interesting to investigate extensions to games where the players move simultaneously. Note that the concept of equilibrium in Definition \ref{def-weakeq}, in particular the notion of consistent run, needs to be adjusted for randomized strategies.

\section*{References}

\noindent
Adlakha S, Johari R, Weintraub GY [2015]. 
Equilibria of dynamic games with many players: Existence, approximation, and market structure. 
\textit{Journal of Economic Theory}, 156, 269-316.\medskip

\noindent Algoet PH and Cover TM [1988]. Asymptotic optimality and asymptotic equipartition properties of log-optimum investment. \textit{The Annals of Probability}, 16(2), 876-898.\medskip

\noindent Al\'{o}s-Ferrer C and Ritzberger K [2016]. \textit{The Theory of Extensive Form Games}. Springer.\medskip
 
\noindent Ashkenazi-Golan G, Flesch J, Predtetchinski A, Solan E [2022]. Existence of equilibria in repeated games with long-run payoffs. \textit{Proceedings of the National Academy of Sciences}, 119(11), e2105867119.\medskip

\noindent
Balseiro SR, Besbes O, Weintraub GY [2015]. 
Repeated auctions with budgets in ad exchanges: Approximations and design. \textit{Management Science}, 61(4), 864-884.\medskip

\noindent Chatterjee K and Henzinger TA [2012]. A survey of stochastic $\omega$-regular games. \textit{Journal of Computer and System Sciences}, 78,394-413 (2012).\medskip

\noindent Dubins LE and Savage LJ [2014]. \textit{How to Gamble if you Must: Inequalities for Stochastic Processes.} New York, Dover Publications. Edited and updated by W. D. Sudderth and D. Gilat.\medskip

\noindent Flesch J, Kuipers J, Mashiah-Yaakovi A, Schoenmakers G, Solan E, Vrieze K [2010]. Perfect-information games with lower-semicontinuous payoffs. \textit{Mathematics of Operations Research}, 35(4), 742-755.\medskip

\noindent Flesch J and Solan E [2023]. Equilibrium in two-player stochastic games with shift-invariant payoffs. \textit{Journal de Mathématiques Pures et Appliquées}, 179, 68-122.\medskip

\noindent Flesch J, Vermeulen D, Zseleva A [2019]. On the equivalence of mixed and behavior strategies in finitely additive decision problems. \textit{Journal of Applied Probability}, 56(3), 810-829.\medskip

\noindent Fudenberg D and Levine D [1983]. Subgame-perfect equilibria of finite-and infinite-horizon games. \textit{Journal of Economic Theory}, 31(2), 251-268.\medskip

\noindent Fudenberg D and Tirole J [1991]. \textit{Game Theory}. MIT press.\medskip

\noindent Gale D and Stewart FM [1953]. Infinite games with perfect information. \textit{Contributions to the Theory of Games}, 2(28), 245-266.\medskip

\noindent Gorokhovsky A and Rubinchik A. [2018]. Regularity of a general equilibrium in a model with infinite past and future. \textit{Journal of Mathematical Economics}, 74, 35-45.\medskip

\noindent
Iyer K, Johari R, Sundararajan  [2014]. 
Mean field equilibria of dynamic auctions with learning. 
\textit{Management Science}, 60(12), 2949-2970.\medskip

\noindent
Kandori M, Mailath GJ, Rob R [1993]. 
Learning, mutation, and long run equilibria in games. 
\textit{Econometrica}, 61(1), 29-56.\medskip

\noindent Kechris A [2012]. \textit{Classical Descriptive Set Theory}. Springer Science \& Business Media.\medskip

\noindent Maitra A and Sudderth W [1998]. Finitely additive stochastic games with Borel measurable payoffs. \textit{International Journal of Game Theory} 27, 257-267.\medskip

\noindent Martin DA [1975]. Borel determinacy, \textit{Annals of Mathematics}, {102}(2), 363-371.\medskip

\noindent Martin DA [1998].
The determinacy of Blackwell games,
\textit{Journal of Symbolic Logic}, {63}(4), 1565-1581.\medskip

\noindent Maschler M, Solan E, Zamir S [2020]. \textit{Game Theory}. Cambridge University Press.\medskip

\noindent Mertens J.-F. [1987]. Repeated games. \textit{Proceedings of the International Congress of Mathematicians}. American Mathematical Society, Providence, RI, 1528-1577.\medskip

\noindent Mertens J.-F., Sorin S, Zamir S [2015]. \textit{Repeated games}. 
Cambridge University Press.\medskip

\noindent Morriston W [1999]. Must the past have a beginning? \textit{Philo}, 2(1), 5-19.\medskip

\noindent Puente S [2006]. Dynamic stability in repeated games, Working Papers 0618, Banco de España.\medskip

\noindent Ray D, Vellodi N, Wang R [2024]
Past and future: backward and forward discounting. \textit{Journal of the European Economic Association}, 22(2), 837-875.\medskip

\noindent Sorensen R [1999]. Infinite ``backward'' induction arguments. \textit{Pacific Philosophical Quarterly}, 80, 278-283.\medskip
    
\end{document}